\documentclass[12pt,a4paper,final]{amsart}
\usepackage{amssymb}
\usepackage{hyperref}
\usepackage[notcite, notref]{showkeys}
\usepackage[utf8]{inputenc}
\usepackage[all]{xy}

\theoremstyle{plain}
\newtheorem{theorem}{Theorem}[section]
\newtheorem{proposition}[theorem]{Proposition}
\newtheorem{corollary}[theorem]{Corollary}
\newtheorem{lemma}[theorem]{Lemma}

\theoremstyle{definition}

\newtheorem{example}[theorem]{Example}
\newtheorem{remark}[theorem]{Remark}
\newtheorem{question}[theorem]{Question}

\newtheorem{problem}[theorem]{Problem}

\theoremstyle{remark}

\numberwithin{equation}{section}
\newcommand{\N}{\mathbb N}
\newcommand{\Z}{\mathbb Z}

\newcommand{\R}{\mathbb R}
\newcommand{\C}{\mathbb C}

\newcommand{\fg}{\mathfrak g}
\newcommand{\fh}{\mathfrak h}

\newcommand{\fk}{\mathfrak k}

\newcommand{\fp}{\mathfrak p}
\newcommand{\fq}{\mathfrak q}

\DeclareMathOperator{\Ot}{O}
\DeclareMathOperator{\SO}{SO}
\DeclareMathOperator{\SU}{SU}

\DeclareMathOperator{\Ut}{U}

\newcommand{\ut}{\mathfrak{u}}

\newcommand{\op}{\operatorname}
\newcommand{\Id}{\textup{Id}}

\newcommand{\mi}{\mathrm{i}}

\DeclareMathOperator{\tr}{Tr}
\DeclareMathOperator{\diag}{diag}
\DeclareMathOperator{\Span}{Span}
\newcommand{\inner}[2]{\langle {#1},{#2}\rangle }
\newcommand{\innerdots}{\langle {\cdot},{\cdot}\rangle }

\DeclareMathOperator{\Ad}{Ad}

\DeclareMathOperator{\Spec}{Spec}
\DeclareMathOperator{\mult}{mult}

\DeclareMathOperator{\vol}{vol}

\newcommand{\HH}{\mathcal{H}}
\newcommand{\PP}{\mathcal{P}}
\newcommand{\GG}{\mathcal{G}}

\renewcommand{\tt}{\theta}

\title{Isospectral CR manifolds with respect to the Kohn Laplacian}

\author{Gerson Gutierrez}
\address{FAMAF--CIEM, Universidad Nacional de C\'ordoba, M.\ Allende s/n, Ciudad Universitaria, 5000 C\'ordoba, Argentina.}
\email{gerson.gutierrez@unc.edu.ar}

\author{Emilio~A.~Lauret}
\address{Instituto de Matem\'atica (INMABB), Departamento de Matem\'atica, Universidad Nacional del Sur (UNS)-CONICET, Bah\'ia Blanca, Argentina.}
\email{emilio.lauret@uns.edu.ar}

\author{Juan Pablo Rossetti}
\address{FAMAF--CIEM, Universidad Nacional de C\'ordoba, M.\ Allende s/n, Ciudad Universitaria, 5000 C\'ordoba, Argentina.}
\email{jprossetti@unc.edu.ar}

\thanks{This research was supported by grants from FONCyT (PICT-2018-02073, PICT-2019-01054), SeCyT-UNC and SGCYT–UNS}
\date{\today}

\begin{document}
\begin{abstract}
We prove that the spectrum of the Kohn Laplacian does not determine the equivalence classes of CR manifolds. 
We construct pairs of odd-dimensional elliptic manifolds that are not equivalent as CR manifolds but whose Kohn Laplacians have the same spectrum. These manifolds are endowed with the CR structures inherited from the canonical CR structure on the sphere of the same dimension.
We provide three different constructions among lens spaces and an additional one among elliptic manifolds with non-cyclic fundamental groups. 
\end{abstract} 

\maketitle

%\tableofcontents

\section{Introduction}
Inverse spectral geometry studies to what extent does the spectral information encode the geometry of an object. 
The case of compact Riemannian manifolds with respect to the spectrum of the canonically associated Laplace-Beltrami operator is the most considered case (see e.g.\ \cite{Gordon00survey}), although the case of drums with respect to the spectrum of the Laplacian with Dirichlet boundary conditions (see \cite{Kac66}) popularized this kind of problems. 

One of the goals of the area is to find isospectral examples, that is, objects with the same spectra but with different shape. 
Such examples are very important to test spectral invariants, that is, properties that are determined by the spectral information. 
Indeed, a property that is not shared by two isospectral objects cannot be spectrally determined, also called \emph{inaudible}. 

In this work we consider the context of CR manifolds with respect to the spectrum of the Kohn Laplacian. 
This operator has been deeply studied, in many articles, and also its spectrum on CR manifolds.
Isospectrality with respect to this operator has been firstly explored, very recently, by Fan, Kim, Plzak, Shors, Sottile and Zeytuncu in \cite{Yunus-cia}.
They dealt with lens spaces (manifolds covered by odd-dimensional spheres with cyclic fundamental groups) endowed with the CR structure induced by the canonical CR structure on the sphere. 
They observed that the order of the fundamental group of a lens space is a spectral invariant. 
Furthermore, they proved that two CR isospectral $3$-dimensional lens spaces with fundamental groups of prime order are necessary equivalent as CR manifolds. 
In other words, two $3$-dimensional lens spaces with fundamental groups of order prime considered as CR manifolds are equivalent if and only if their corresponding Kohn Laplacians have the same spectrum.

The question whether there are examples of CR isospectral manifolds remained open. 
The main goal of this article is to provide different constructions of CR isospectral elliptic manifolds (those covered by odd-dimensional spheres with finite fundamental groups) endowed with the CR structures induced by the canonical CR structure on the sphere, which are not CR diffeomorphic. 
To do that, we extend techniques from the Riemannian case developed by Ikeda (e.g.\ \cite{IkedaYamamoto79}, \cite{Ikeda80_isosp-lens}, \cite{Ikeda80_3-dimI}, \cite{Ikeda80_3-dimII}, \cite{Ikeda83}), also from \cite{LMR-onenorm} and \cite{DD18}.
See \cite{LMR-SaoPaulo} for a recent account on the spectra of lens spaces as Riemannian manifolds with their round metrics. 

Besides introducing preliminaries on Riemannian and CR structures on elliptic manifolds, in Section~\ref{sec:preliminaries} we define the main tool of the article: a two variable generating function associated to the spectrum of the Kohn Laplacian over an elliptic CR manifold. 
A rational expression is given in Theorem~\ref{thm:F_Gamma(z,w)}. 
All isospectral constructions in the rest of the article are done by showing that their corresponding generating functions coincide. 

The first isospectral examples are given in Section~\ref{sec:ikeda}.
For each odd prime number $k$, we construct a family of lens spaces with fundamental groups of order $k$ that are mutually CR isospectral and such that the number of equivalence classes as CR manifolds in that family can be arbitrary large as long as $k$ is large enough (Theorem~\ref{thm:Gerson} and Proposition~\ref{prop:arbitraryliylarge}).

In Section~\ref{sec:computational} we establish a finiteness condition to find CR isospectral lens spaces (Theorem~\ref{thm:finitecondition}). 
This allows us to find particular examples with the help of a computer for low values of $n$ and $k$, where $k$ 
is the order of the fundamental groups and $2n-1$ is the dimension (Table~\ref{table:isosp}).

Section~\ref{sec:CRisosppairs} contains the third isospectral construction. 
Theorem~\ref{thm:theoremn} provides an enormous family of pairs of CR isospectral CR lens spaces.
This is done in a similar way as in \cite{DD18}, by DeFord and Doyle.
As a consequence, we obtain in each odd dimension infinitely many pairs of CR isospectral non equivalent CR lens spaces.  
Moreover, many of the computational examples belong to this construction. 

The three previously mentioned constructions deal with lens spaces. 
Section~\ref{sec:isospsphericalspaceforms} gives examples of CR isospectral elliptic manifolds with non-cyclic fundamental groups. 

It turned out that the generating function used in all of the above constructions also encodes the spectrum of the Laplace-Beltrami operator of an elliptic manifold endowed with certain non-round Riemannian metrics, namely, the metrics induced by Berger spheres. 
In Section~\ref{sec:Berger} we explain that every single CR isospectral pair given in the previous sections produces a continuous curve of isospectral pairs of Riemannian manifolds in the classical sense.

There are also several questions and open problems disseminated throughout the article. 

\subsection*{Acknowledgments}
The authors wish to thank professors John D'Angelo, 
 Adrián Andrada and Roberto Miatello for helpful comments.

\section{Preliminaries}\label{sec:preliminaries}

This section introduces the preliminaries necessary to work in the subsequent sections. 
It starts with basic notions of spherical space forms and the canonical CR structures admitted by their underlying differentiable manifolds. 
It continues with the important subclass of spherical space forms having cyclic fundamental groups called lens spaces. 
It ends reviewing the spectra of the Kohn Laplacian on the spheres and their quotients.

\subsection{Riemannian and CR structures on elliptic manifolds}\label{subsec:RiemCR}

The group $\Ut(n)$ embeds into $\SO(2n)$ by sending each coordinate $z=a+\mi b\in\C$ to $\left(\begin{smallmatrix} a&b\\ -b& a\end{smallmatrix}\right)\in \textup{M}_{2\times2}(\R)$. 
The group $\SO(2n)$ acts on the unit sphere $S^{2n-1}$ in $\R^{2n}$ by left multiplication; one can easily see that this action is transitive, as well as its restriction to $\Ut(n)$. 

Given $\Gamma$ a finite subgroup of $\Ot(d+1)$ acting freely on $S^{d}$ (i.e.\ $\gamma\cdot x=x$ with $\gamma \in\Gamma$ and $x\in S^{d}$ occurs only if $\gamma=I$), the corresponding manifold $S^{d}/\Gamma$ is called \emph{elliptic}.
We will restrict our attention to odd dimensional elliptic manifolds, which turn out to be orientable, and consequently, any of them is of the form $S^{2n-1}/\Gamma$ with $\Gamma\subset\SO(2n)$ acting freely on $S^{2n-1}$. 
Moreover, from their classification (see \cite{Wolf-book}), it turns out that $\Gamma$ can be assumed to be inside $\Ut(n)$. 

A \emph{spherical space form} is a complete Riemannian manifold with constant positive sectional curvature. 
Any of them is isometric to an elliptic manifold $S^{2n-1}/\Gamma$ endowed with a round metric (i.e.\ the metric induced by the canonical metric of constant sectional curvature equal to $1$ in $S^{2n-1}$). 
The next result is classic (see e.g.\ \cite[Lem.~5.1.1]{Wolf-book}).

\begin{theorem}\label{thm:isometries}
Let $\Gamma,\Gamma'$ be finite subgroups of $\Ut(n)$ acting freely on $S^{2n-1}$. 
The spherical space forms $S^{2n-1}/\Gamma$ and $S^{2n-1}/\Gamma'$ are isometric if and only if $\Gamma$ and $\Gamma'$ are conjugate in $\Ot(2n)$. 
\end{theorem}

A \emph{Cauchy-Riemann manifold}, abbreviated as \emph{CR manifold}, is an odd-dimensional differentiable manifold $M$ together with a CR structure. 
If $\dim M=2n-1$, a \emph{CR structure} is a subbundle $T^{(1,0)}(M)$ of the complexified tangent bundle $T_\C(M)$ of (complex) dimension $n-1$ satisfying that $T^{(1,0)}(M)\cap \overline{T^{(1,0)}(M)}=0$ and %the integrability condition 
$[\Gamma(U,T^{(1,0)}(M)), \Gamma(U,T^{(1,0)}(M))]\subset \Gamma(U,T^{(1,0)}(M))$, for any open subset $U$ of $M$. 

Two CR manifolds $(M,T^{(1,0)}(M))$ and $(N,T^{(1,0)}(N))$ are called \emph{CR diffeomorphic} (as in \cite{ChenShaw})  or just \emph{equivalent}
if there is a diffeomorphism $F:M\to N$ satisfying $dF(T^{(1,0)}(M))=T^{(1,0)}(N)$. 

The unit sphere $S^{2n-1}$ has a canonical CR structure, namely $T^{(1,0)}S^{2n-1}:= T^{(1,0)}(\C^{n})\cap T_\C(S^{2n-1})$, where $T^{(1,0)}(\C^{n})=\Span_\C\{\frac{\partial}{\partial z_1},\dots, \frac{\partial}{\partial z_n}\}$, the complex vectors annihilating the antiholomorphic functions.
Here, we are seeing $S^{2n-1}$ as the unit sphere in $\C^n$. 
Given $\Gamma$ a finite subgroup of $\Ut(n)$ acting freely on $S^{2n-1}$, the elliptic manifold $S^{2n-1}/\Gamma$ inherits a canonical CR-structure from $S^{2n-1}$.
Although the manifold $S^{2n-1}/\Gamma$ admits other CR structures, we will always assume the canonical one, and we will call it an \emph{elliptic CR manifold}. 

The next result might be well known, but we include a sketch of the proof because the authors were not able to find it stated as below. 

\begin{theorem}\label{thm:CRequiv}
Let $\Gamma,\Gamma'$ be finite subgroups of $\Ut(n)$ acting freely on $S^{2n-1}$. 
The elliptic CR manifolds $S^{2n-1}/\Gamma$ and $S^{2n-1}/\Gamma'$ are equivalent if and only if $\Gamma$ and $\Gamma'$ are conjugate in $\Ut(n)$. 
\end{theorem}

\begin{proof}[Sketch of proof]
We abbreviate $M=S^{2n-1}/\Gamma$ and $M'=S^{2n-1}/\Gamma'$, and let $\pi:S^{2n-1}\to M$ and $\pi':S^{2n-1}\to M'$ be the natural projections.  
Suppose that $F:(M,T^{(1,0)}(M)) \to (M',T^{(1,0)}(M'))$ is a CR diffeomorphism. 
This map induces $\widetilde F:S^{2n-1}\to S^{2n-1}$ such that the following diagram commutes:
\begin{equation}\label{eq:diagram}
\xymatrix{
	S^{2n-1}  \ar[r]^{\widetilde F} \ar[d]_\pi& S^{2n-1}\ar[d]^{\pi'}
	\\
	M\ar[r]^{F}&M'
}.
\end{equation}
It follows that $\Gamma'=\{\widetilde F\circ L_\gamma\circ \widetilde F^{-1}: \gamma\in\Gamma\}$, where $L_{\gamma}(z)= \gamma z$.

It turns out that $\widetilde F:(S^{2n-1}, T^{(1,0)})\to (S^{2n-1}, T^{(1,0)})$ is a CR automorphism.
It is well known that the set of CR automorphism of $(S^{2n-1}, T^{(1,0)})$ is given by $\op{PU}(n,1)$ acting by Möbius action on the unit ball $B^n$ of $\C^n$, and also on its boundary $S^{2n-1}$. 
However, $\widetilde F$ sends the origin to itself because $B^n/\Gamma$ and $B^n/\Gamma'$ have both a single singularity at the origin, and this forces that $\widetilde F$ lies in the isotropy subgroup of $\op{PU}(n,1)$ at the origin. 
Consequently, there is $A\in \Ut(n)$ such that $\widetilde F(z)=Az$. 
Hence $A\Gamma A^{-1}=\Gamma'$ as asserted.

To show the converse, if $A\Gamma A^{-1}=\Gamma'$ for some $A\in\Ut(n)$, then the map $\widetilde F:\C^n\to \C^n$ given by $\widetilde F(z)= Az$ restricted to $S^{2n-1}$ becomes an automorphism of $(S^{2n-1}, T^{(1,0)}(S^{2n-1}))$ and descends to a CR diffeomorphism $F:(M,T^{(1,0)}(M))\to (M',T^{(1,0)}(M'))$, concluding that the CR manifolds $M$ and $M'$ are equivalent.
\end{proof}

The information above is the only fact that we will use about Riemannian and CR geometry on elliptic manifolds. 
For more details, the reader may consult for instance \cite{ChenShaw,DAngeloLichtblau92,Wolf-book}.

\begin{corollary}\label{cor:isom=>CRequiv}
Let $\Gamma,\Gamma'$ be finite subgroups of $\Ut(n)$ acting freely on $S^{2n-1}$. 
If the elliptic CR manifolds $S^{2n-1}/\Gamma$ and $S^{2n-1}/\Gamma'$ are equivalent, then the corresponding spherical space forms are isometric. 
\end{corollary}

\begin{proof}
Theorem~\ref{thm:CRequiv} implies that $\Gamma$ and $\Gamma'$ are conjugate in $\Ut(n)$ by Theorem~\ref{thm:CRequiv}. 
Hence, $\Gamma$ and $\Gamma'$ are also conjugate in $\Ot(2n)$, which gives the required isometry by Theorem~\ref{thm:isometries}. 
\end{proof}

\begin{remark}\label{rem:converse-isom=>CRequiv}
The converse of Corollary~\ref{cor:isom=>CRequiv} does not hold. 
There are isometric spherical space forms $S^{2n-1}/\Gamma$ and $S^{2n-1}/\Gamma'$ such that the corresponding CR manifolds are not equivalent. 
In other words, there exist finite subgroups $\Gamma,\Gamma'$ of $\Ut(n)$ acting freely on $S^{2n-1}$ that are conjugate in $\Ot(2n)$ but not in $\Ut(n)$. 
Remark~\ref{rem:lensspaces} includes an explicit example of this situation. 
\end{remark}

\subsection{Lens spaces}

For $k\in\N$ and $s=(s_1,\dots,s_n)\in\Z^n$ satisfying that $\gcd(k,s_i)=1$ for all $i$, we denote by $\Gamma_{k,s}$ the cyclic subgroup of $\Ut(n)$ generated by 
\begin{equation}
\gamma_{k,s} :=
\begin{pmatrix}
\xi_k^{s_1} \\ &\ddots \\ && \xi_k^{s_n}
\end{pmatrix},
\end{equation}
where $\xi_k=e^{\frac{2\pi\mi }{k}}$.
Under the embedding $\Ut(n)\hookrightarrow \SO(2n)$, $\gamma_{q,s}$ corresponds to 
\begin{equation}
\begin{pmatrix}
R(\frac{2\pi s_1}{k})\\ 
&\ddots\\
&& R(\frac{2\pi s_n}{k})
\end{pmatrix}
,
\end{equation}
where
$
R(\theta)=
\left(\begin{smallmatrix}
\cos \theta & \sin \theta\\
-\sin \theta & \cos\theta 
\end{smallmatrix}\right).
$

It turns out that $\Gamma_{k,s}$ acts freely on $S^{2n-1}$. 
The quotient $L(k;s):=S^{2n-1}/\Gamma_{k,s}$ is called a \emph{lens space}. 
The fundamental group of $L(k;s)$ is cyclic of order $k$. 
The standard differentiable structure on $S^{2n-1}$ induces a canonical differentiable structure on every lens space. 
According to the convention in Subsection~\ref{subsec:RiemCR}, we will call $L(k;s)$ a \emph{Riemannian lens space} when it is endowed with the canonical Riemannian metric of constant sectional curvature one, and a \emph{CR lens space} when it is endowed with the canonical CR structure.

The following result is well known. 

\begin{proposition}\label{prop:isometria}
Let $L(k;s)$ and $L(k;s')$ be two lens spaces of dimension $2n-1$. 
The following assertions are equivalent: 
\begin{enumerate}
\item $\Gamma_{k,s}$ and $\Gamma_{k,s'}$ are conjugate in $\Ot(2n)$. 

\item $L(k;s)$ and $L(k;s')$ are homeomorphic. 

\item $L(k;s)$ and $L(k;s')$ are diffeomorphic (with their canonical differentiable structures). 

\item $L(k;s)$ and $L(k;s')$ are isometric as Riemannian lens spaces. 

\item There are $\sigma$ a permutation of $\{1,\dots,n\}$, $\epsilon_i\in\{\pm1\}$ for each $i=1,\dots,n$, and $c\in\Z$ prime to $k$ such that 
\begin{equation}
s_i \equiv \epsilon_i c s'_{\sigma(i)} \pmod k
\qquad\text{for all }i=1,\dots,n. 
\end{equation}
\end{enumerate}
\end{proposition}

The analogous result in the context of CR manifolds is the following.

\begin{proposition}\label{prop:CR-isometria}
Let $L(k;s)$ and $L(k;s')$ be two lens spaces of dimension $2n-1$. 
The following assertions are equivalent:
\begin{enumerate}
\item $\Gamma_{k,s}$ and $\Gamma_{k,s'}$ are conjugate in $\Ut(n)$. 

\item $L(k;s)$ and $L(k;s')$ are equivalent as CR lens spaces. 

\item there are $\sigma$ a permutation of $\{1,\dots,n\}$, and $c\in\Z$ prime to $k$ such that 
\begin{equation}
s_i \equiv c s'_{\sigma(i)} \pmod k
\qquad\text{for all }i=1,\dots,n. 
\end{equation}
\end{enumerate}
\end{proposition}

\begin{proof}
The equivalence between (1) and (2) follows immediately by Theorem~\ref{thm:CRequiv}. 
Suppose that $\Gamma_{k,s}$ and $\Gamma_{k,s'}$ are conjugate in $\Ut(n)$.
Thus $\gamma_{k,s} = A\gamma_{k,s'}^c A^{-1}$ for some $A\in \Ut(n)$ and $c\in\Z$.
Note that the spectra of $\gamma_{k,s}$ and $\gamma_{k,s'}^c$ coincide, thus $(c,k)=1$ and, there is a permutation $\sigma$ such that $\xi_k^{s_i} =\xi_{k}^{cs'_{\sigma(i)}}$ for all $1\leq i\leq n$, which implies the required condition in (3). 
Finally, (3)$\Rightarrow$(1) is an easy exercise (see also \cite[Thm.~4.13]{Yunus-cia}). 
\end{proof}

\begin{remark}\label{rem:lensspaces}
It is easy to conclude from Propositions~\ref{prop:isometria}--\ref{prop:CR-isometria} that $L(3;1,1)$ and $L(3;1,2)$ are isometric as Riemannian lens spaces, but not equivalent as CR lens spaces. 
\end{remark}

\subsection{Kohn Laplacian on elliptic CR manifolds}\label{subsec:Kohnelliptic}
An arbitrary CR manifold has associated a distinguished differential operator, the \emph{Kohn Laplacian}. 
Two arbitrary CR manifolds are called \emph{CR isospectral} if the spectra of their corresponding Kohn Laplacian coincide. 
Our attention restrict to the spectrum of the Kohn Laplacian on a elliptic CR manifolds, that we next describe.

For $p,q\in\N_0$, let 
\begin{equation}\label{eq:PqpHpq}
\begin{aligned}
\PP_{p,q} &= \Span_\C\left\{ z_1^{a_1}\dots z_n^{a_n} \bar z_1^{b_1} \dots \bar z_n^{b_n} :  
a_i,b_i\in\N_0\;\forall i,\; {\textstyle \sum\limits_{i=1}^n a_i}=p, {\textstyle \sum\limits_{i=1}^n b_i}=q \right\}
,\\ 
\HH_{p,q} &=\{P \in \PP_{q,p}: \Delta P=0\},
\end{aligned}
\end{equation}
where $\Delta=\sum_{j=1}^n \frac{\partial^2}{\partial z_j\partial \bar z_j}$. 
In words, $\HH_{p,q}$ denotes the space of harmonic polynomials on the complex variables $z_1,\dots,z_n,\bar z_1,\dots,\bar z_n$ of bi-degree $(p,q)$. 
One has 
\begin{equation}\label{eq:Ppq=Hpq+|z|^2Pp-1q-1}
\PP_{p,q}=\HH_{p,q}\oplus (z_1\bar z_1+\dots+z_n\bar z_n)\PP_{p-1,q-1}. 
\end{equation}

It is well known that 
\begin{equation}\label{eq:L^2(S^2n-1)}
L^2(S^{2n-1})= \widehat{\bigoplus_{p,q\in\N_0}} \HH_{p,q}, 
\end{equation}
when the elements in $\HH_{p,q}$ are understood as complex-valued functions on $S^{2n-1}$. 
Folland~\cite{Folland} described the spectrum of the Kohn Laplacian $\square_b$ on the sphere $S^{2n-1}$ as follows: 
$\HH_{p,q}$ is included in the eigenspace of $\square_b$ associated to the eigenvalue $\lambda_{p,q}:=2q(p+n-1)$. 
Consequently, every eigenvalue of $\square_b$ is of the form $2r$ with $r\in\Z$ and $r\geq n-1$, and the multiplicity of $2r$ is given by 
\begin{equation}
\mult(2r) = \sum_{p,q\geq0 \,: \,   2r=2q(p+n-1)} \dim\HH_{p,q}. 
\end{equation}

The action of $\Ut(n)$ on $S^{2n-1}$ preserves the CR-structure and commutes with $\square_b$. 
Moreover, the left regular representation of $\Ut(n)$ on $L^2(S^{2n-1})$ given by $(A\cdot f)(z)=f(A^{-1}z)$ for $A\in\Ut(n)$ and $f\in L^2(S^{2n-1})$, decomposes in irreducible modules as in \eqref{eq:L^2(S^2n-1)}.

Let $\Gamma$ be a subgroup of $\Ut(n)$ acting freely on $S^{2n-1}$. 
We denote by $\square_{b,\Gamma}$ the Kohn Laplacian on the elliptic CR manifold $S^{2n-1}/\Gamma$. 
Clearly, any eigenfunction in $\HH_{p,q}$ which is invariant by $\Gamma$ descends to an eigenfunction on $S^{2n-1}/\Gamma$ with the same eigenvalue. 
Moreover, we have that 
\begin{equation}\label{eq:L^2=sumHpq}
L^2(S^{2n-1}/\Gamma)=\widehat{\bigoplus_{p,q\geq0}} \HH_{p,q}^\Gamma
\end{equation}
where $\HH_{p,q}^\Gamma=\{f\in\HH_{p,q}:  \gamma \cdot f=f\ \ \forall \,\gamma\in\Gamma\}$.
Therefore, the multiplicity of the eigenvalue $2r$ of $\square_{b,\Gamma}$ on $S^{2n-1}/\Gamma$ is given by 
\begin{equation}\label{eq:mult_Gamma(2r)}
\mult_{\Gamma}(2r) := \sum_{{p,q\geq0 \, : \, 2r=2q(p+n-1)}}  \dim\HH_{p,q}^\Gamma. 
\end{equation}
Note that this number might be zero, which means that $2r$ is not in the spectrum $\Spec(\square_{b,\Gamma})$ of $\square_{b,\Gamma}$, though it is an eigenvalue of $\square_{b,\Gamma}$ on $S^{2n-1}$.

\begin{proposition}\label{prop2:Hisop=>CRisosp} 
Let $\Gamma,\Gamma'$ be finite subgroups of $\Ut(n)$ acting freely on $S^{2n-1}$.
If
\begin{equation}\label{eq:dimHpq^Gamma=dimHpq^Gamma'}
\dim\HH_{p,q}^{\Gamma} = \dim\HH_{p,q}^{\Gamma'}
\quad\text{for all }p,q\geq0, 
\end{equation}
then the elliptic CR manifolds $S^{2n-1}/\Gamma$ and $S^{2n-1}/\Gamma'$ are CR isospectral, that is, $\Spec(\square_{b,\Gamma})=\Spec(\square_{b,\Gamma'})$. 
\end{proposition}

\begin{remark}\label{rem:convese}
It is not known whether the converse is true or not. 
In other words, the condition \eqref{eq:dimHpq^Gamma=dimHpq^Gamma'} is sufficient for CR isospectrality, but it is not clear whether it is necessary as well. 
\end{remark}

Mimicking Ikeda's approach to study isospectrality among spherical space forms (in the Riemannian case), we set
\begin{equation}\label{eq:F}
F_{\Gamma}(z,w)=\sum_{p,q\geq0} \dim\HH_{p,q}^\Gamma \, z^pw^q, 
\end{equation}
where $z,w$ are formal variables (\cite[\S4.2]{Yunus-cia} defines it for the special case of CR lens spaces).
Taking into account that \eqref{eq:dimHpq^Gamma=dimHpq^Gamma'} holds if and only if $F_{\Gamma}(z,w)=F_{\Gamma'}(z,w)$, Proposition~\ref{prop2:Hisop=>CRisosp} translates into the following result, which will be extremely important in the rest of the article. 

\begin{proposition}\label{prop:F_Gamma=>CR-isosp}
Let $\Gamma,\Gamma'$ be finite subgroups of $\Ut(n)$ acting freely on $S^{2n-1}$. 
If $F_{\Gamma}(z,w) = F_{\Gamma'}(z,w)$, then $S^{2n-1}/\Gamma$ and $S^{2n-1}/\Gamma'$ are CR isospectral.
\end{proposition}

Although $F_{\Gamma}(z,w)$ determines the spectrum of the Kohn Laplacian on $S^{2n-1}/\Gamma$, the converse is not necessarily true according to Remark~\ref{rem:convese}. 
The following question is open.

\begin{question}
Are there finite subgroups $\Gamma,\Gamma'$ of $\Ut(n)$ acting freely on $S^{2n-1}$ such that the elliptic CR manifolds $S^{2n-1}/\Gamma$ and $S^{2n-1}/\Gamma'$ are CR isospectral but $F_{\Gamma}(z,w)\neq F_{\Gamma'}(z,w)$?
\end{question}

Of course, if two CR manifolds are equivalent, then they are CR isospectral (see for instance \cite[\S4.3]{Yunus-cia}). 
The next result shows that the CR isospectrality between equivalent elliptic CR manifolds is necessarily obtained via Proposition~\ref{prop2:Hisop=>CRisosp}. 

\begin{proposition}
Let $\Gamma,\Gamma'$ be finite subgroups of $\Ut(n)$ acting freely on $S^{2n-1}$. 
If $S^{2n-1}/\Gamma$ and $S^{2n-1}/\Gamma'$ are equivalent, then $F_{\Gamma}(z,w)=F_{\Gamma'}(z,w)$. 
\end{proposition}

\begin{proof}
From Theorem~\ref{thm:CRequiv}, we have that $\Gamma$ and $\Gamma'$ are conjugate in $\Ut(n)$ due to the equivalence between $S^{2n-1}/\Gamma$ and $S^{2n-1}/\Gamma'$.
There is $A\in \Ut(n)$ such that $A\Gamma A^{-1}=\Gamma'$. 
Hence
\begin{equation*}
\begin{aligned}
\dim \HH_{p,q}^{\Gamma'} &
= \dim \HH_{p,q}^{A\Gamma A^{-1}} 
= \dim \HH_{p,q}^{\Gamma} .
\end{aligned}
\end{equation*}
The last identity follows since the map $\HH_{p,q}^{A\Gamma A^{-1}}\to \HH_{p,q}^{\Gamma}$, sending $f\mapsto A^{-1}\cdot f$, is an isomorphism of vector spaces. 
\end{proof}

\section{A rational expression for the generating function}\label{sec:rationalexpression}

The main goal of this section is to give a rational expression of the generating function $F_{\Gamma}(z,w)$ associated to the spectrum of the Kohn Laplacian of the elliptic CR manifold $S^{2n-1}/\Gamma$ introduced in \eqref{eq:F}. 
Such expression extends to elliptic CR manifolds the formula obtained in \cite[Thm.~4.10]{Yunus-cia}, which is valid for CR lens spaces. 
The proof follows the one by Ikeda in \cite[Thm.~2.2]{Ikeda80_3-dimI}, which proves an analogous formula in the Riemannian case. 

\begin{theorem}\label{thm:F_Gamma(z,w)}
For $\Gamma$ a finite subgroup of $\Ut(n)$ acting freely on $S^{2n-1}$, we have that 
\begin{equation}
F_{\Gamma}(z,w) = \frac1{|\Gamma|} \sum_{\gamma\in\Gamma} \frac{1-zw}{\det(z-\gamma) \det(w-\bar \gamma)}
,
\end{equation} 
where $\det(x-\gamma)=\prod_{\lambda} (x-\lambda)$ with $\lambda$ running over the set of eigenvalues of $\gamma$. 
\end{theorem}

\begin{proof}
Write $k=|\Gamma|$. 
It is well known that the linear operator $\Psi_\Gamma:\PP_{p,q}\to \PP_{p,q}$ given by $\Psi_\Gamma(P)=\frac1k \sum_{\gamma\in\Gamma} \gamma\cdot P$ preserves $\HH_{p,q}$, $\Phi_{\Gamma}(\HH_{p,q})= \HH_{p,q}^\Gamma$, and $\Psi_{\Gamma}|_{\HH_{p,q}^\Gamma}=\Id_{\HH_{p,q}^\Gamma}$, therefore $\dim \HH_{p,q}^\Gamma=\tr(\Psi_\Gamma)$. 
Let us denote by $\chi_{p,q}$ and $\widetilde\chi_{p,q}$ the characters of the $\Ut(n)$-modules $\HH_{p,q}$ and $\PP_{p,q}$ respectively. 
Thus, $\tr(\Psi_\Gamma)=\frac1k \sum_{\gamma\in\Gamma} \tr(P\mapsto \gamma\cdot P)= \frac1k \sum_{\gamma\in\Gamma} \chi_{p,q}(\gamma)$. 
One has that $\chi_{p,q}=\widetilde \chi_{p,q} - \widetilde \chi_{p-1,q-1}$ by \eqref{eq:Ppq=Hpq+|z|^2Pp-1q-1}, thus 
\begin{equation}\label{eq:F_Gamma-characters}
\begin{aligned}
F_{\Gamma}(z,w) &
=\sum_{p,q\geq0} \dim \HH_{p,q}^\Gamma z^pw^q
=\frac1k \sum_{p,q\geq0} \sum_{\gamma\in\Gamma} \chi_{p,q}(\gamma) z^pw^q
\\ & 
=\frac1k \sum_{\gamma\in\Gamma} \sum_{p,q\geq0}  \big(\widetilde \chi_{p,q}(\gamma) -\widetilde \chi_{p-1,q-1}(\gamma)\big) z^pw^q
\\ & 
=\frac1k \sum_{\gamma\in\Gamma} \left(
	\sum_{p,q\geq0}  \widetilde \chi_{p,q}(\gamma) z^pw^q
	-
	zw\sum_{p,q\geq0} \widetilde \chi_{p,q}(\gamma) z^{p}w^{q}
\right)
\\ & 
=\frac{1-zw}k \sum_{\gamma\in\Gamma} 
	\sum_{p,q\geq0}  \widetilde \chi_{p,q}(\gamma) z^pw^q
.
\end{aligned}
\end{equation}

Let $\gamma\in\Gamma$. 
Since $\gamma^k=I$, the eigenvalues of $\gamma$ are of the form $\xi_k^{\ell_1},\dots, \xi_k^{\ell_n}$ with $\ell_i\in\Z$ for all $i$, where $\xi_k=e^{\frac{2\pi\mi}{k}}$.
Furthermore, $\chi_{p,q}(\gamma)=\chi_{p,q}(\gamma')$ with $\gamma'=\diag(\xi_k^{\ell_1}, \dots, \xi_k^{\ell_n})$. 
A basis of $\PP_{p,q}$ is given by $\{\prod_{i=1}^n z_i^{\alpha_i}\bar z_i^{\beta_i} :  \alpha_i,  \beta_i\in\N_0\;\forall i,\; \alpha_1+\dots+\alpha_n=p,\; \beta_1+\dots+\beta_n=q\}$  and
\begin{equation*}
\gamma'\cdot \prod_{i=1}^n z_i^{\alpha_i}\bar z_i^{\beta_i} 
= \prod_{i=1}^n (\xi_k^{-\ell_i} z_i)^{\alpha_i} (\overline{\xi_k^{-\ell_i}z_i})^{\beta_i}
= \xi_k^{-\sum_{i=1}^n \ell_i(\alpha_i-\beta_i)} \prod_{i=1}^n z_i^{\alpha_i} \bar z_i^{\beta_i}.
\end{equation*}
Hence  
\begin{equation*}
\begin{aligned}
\sum_{p,q\geq0}\widetilde \chi_{p,q}(\gamma) z^p w^q&
= \sum_{p,q\geq0}\widetilde \chi_{p,q}(\gamma') z^p w^q 
= \sum_{p,q\geq0} \sum_{\substack{\alpha_1+\dots+a_n=p\\ \beta_1+\dots+\beta_n=q}} \xi_k^{-\sum_i \ell_i(\alpha_i-\beta_i)} z^p w^q
\\ & 
= \sum_{\alpha\in\N_0^n} \sum_{\beta\in\N_0^n} \xi_k^{-\sum_i \ell_i \alpha_i} z^{\alpha_1+\dots+\alpha_n} \xi_k^{\sum_i \ell_i\beta_i} w^{\beta_1+\dots+\beta_n}
\\ & 
= \sum_{\alpha\in\N_0^n} \sum_{\beta\in\N_0^n} \prod_{i=1}^n (\xi_k^{-\ell_i} z)^{\alpha_i} (\xi_k^{\ell_i} w)^{\beta_i}
%
%
%\\ & 
= \prod_{i=1}^n \left(\sum_{\alpha_i\in\N_0}(\xi_k^{-\ell_i} z)^{\alpha_i}\right) \left( \sum_{\beta_i\in\N_0}   (\xi_k^{\ell_i} w)^{\beta_i}\right)
\\ & 
= \prod_{i=1}^n \frac{1}{(1-\xi_k^{-\ell_i}z) (1-\xi_k^{\ell_i}w)}
%
%
%\\ & 
= \frac{1}{ \det(z-\gamma) \det(w-\bar\gamma)}
,
\end{aligned}
\end{equation*}
which combined with \eqref{eq:F_Gamma-characters} concludes the proof.  
\end{proof}

\begin{remark}
Instead of $F_{\Gamma}(z,w)$, one can associated to an elliptic CR manifold $S^{2n-1}/\Gamma$ 
the \emph{authentical} generating function 
\begin{equation*}
G_\Gamma(z)
=\sum_{r\geq n} \op{mult}_{\Gamma}(2r)\, z^r
= \sum_{r\geq n} z^r \sum_{ \substack{p,q\geq0 \, : \\ 2r=2q(p+n-1)}}  \dim\HH_{p,q}^\Gamma 
,
\end{equation*}
where again $z$ is a formal variable. 
Recall from \eqref{eq:mult_Gamma(2r)} that $\op{mult}_{\Gamma}(2r)$ denotes the multiplicity of $2r$ in the spectrum of the Kohn Laplacian $\square_{b,\Gamma}$ on $S^{2n-1}/\Gamma$. 
Consequently, unlike $F_{\Gamma}(z,w)$, the generating function $G_{\Gamma}(z)$ actually encodes $\Spec(\square_{b,\Gamma})$ in the sense that, given two elliptic CR manifolds $S^{2n-1}/\Gamma$ and $S^{2n-1}/\Gamma'$, one has that
\begin{equation*}
\Spec(\square_{b,\Gamma})=\Spec(\square_{b,\Gamma'}) 
\quad\Longleftrightarrow \quad 
G_{\Gamma}(z)=G_{\Gamma'}(z). 
\end{equation*} 
However, the authors were not able to find a nice expression for $G_{\Gamma}(z)$ like in Theorem~\ref{thm:F_Gamma(z,w)}, 
so the generating function $G_\Gamma(z)$ has no applications so far. 
\end{remark}

It follows from Theorem~\ref{thm:F_Gamma(z,w)} that the generating function $F_{L(k;s)}(z,w):= F_{\Gamma_{k,s}}(z,w)$ for a CR lens space $L(k;s)$ is given by 
\begin{equation}\label{eq:F_L(z,w)}
F_{L(k;s)}(z,w) =\frac{1}{k} \sum_{m=0}^{k-1} \frac{1-zw}{\prod_{i=1}^{n} (z-\xi_k^{-s_im})(w-\xi_k^{s_im})}.
\end{equation}
This expression was established in \cite[Thm.~4.10]{Yunus-cia} and will be the main tool in the three constructions of CR isospectral CR lens spaces of the next sections. 
 
We end this section with a connection between CR isospectral elliptic CR manifolds via Proposition~\ref{prop:F_Gamma=>CR-isosp} and isospectral spherical space forms with respect to the Laplace-Beltrami operator. 

\begin{remark}\label{rem:H_lH_pq}
Similarly as in  
\eqref{eq:L^2(S^2n-1)}, it is well known that $L^2(S^{2n-1})\simeq \bigoplus_{l\geq0}\widetilde \HH_l$ as $\SO(2n)$-modules, where $\widetilde \HH_l$ stands for the space of harmonic homogeneous polynomials in the variables $x_i$ for $i=1,\dots,2n$ of degree $l$, which is irreducible as a representation of $\SO(2n)$, and is precisely the eigenspace of the Laplace-Beltrami operator on the round sphere $S^{2n-1}$ with eigenvalue $k(k+2n-2)$. 
Moreover, for $\Gamma$ a finite subgroup of $\SO(2n)$ acting freely on $S^{2n-1}$ (which can be assumed inside $\Ut(n)$), the spectrum of the Laplace-Beltrami operator associated to the round metric on $S^{2n-1}/\Gamma$ is determined as follows: 
the multiplicity of $k(k+2n-2)$ is given by $\dim\widetilde\HH_l^\Gamma$.

By identifying the variables $z_j=x_{2j-1}+\mi x_{2j}$ and $\bar z_j=x_{2j-1}-\mi x_{2j}$ for all $j=1,\dots,n$, we obtain the following isomorphism of $\Ut(n)$-modules: 
\begin{equation*}
\widetilde \HH_l\simeq \bigoplus_{p,q\geq0\,: \, l=p+q} \HH_{p,q}. 
\end{equation*}
Consequently, $\dim\widetilde \HH_l^\Gamma = \sum\limits_{p=0}^l \dim \HH_{p,l-p}^\Gamma. 
$
\end{remark}

\begin{theorem}\label{thm:F_Gamma=>Riem-isosp}
Let $\Gamma,\Gamma'$ be finite subgroups of $\Ut(n)$ acting freely on $S^{2n-1}$. 
If $F_{\Gamma}(z,w) = F_{\Gamma'}(z,w)$, then  the spherical space forms $S^{2n-1}/\Gamma$ and $S^{2n-1}/\Gamma'$ are isospectral with respect to the Laplace-Beltrami operator.
\end{theorem}

\begin{proof}
It follows immediately from Remark~\ref{rem:H_lH_pq} that
\begin{equation*}
F_{\Gamma}(z,z)
=\sum_{l \geq0} \left(\sum_{p,q\geq0\,: \, l=p+q} \dim\HH_{p,q}^\Gamma \right) z^{l}
=\sum_{l \geq0} \dim\widetilde\HH_l^\Gamma\, z^{l}
. 
\end{equation*}
The expression on the right-hand side above is precisely Ikeda's generating function associated to the spherical space form $S^{2n-1}/\Gamma$.
The claim follows since Ikeda proved in \cite[Prop.~2.1]{Ikeda80_3-dimI} that two spherical space forms are isospectral if and only if their generating functions coincide. 
\end{proof}

\section{Arbitrarily large mutually CR isospectral family} \label{sec:ikeda}
In this section we construct for each odd prime number $k$, a pairwise CR isospectral family of CR lens spaces with fundamental group of order $k$.
The size of such family is arbitrarily large when $k$ increases. 
The method is a slight variation of Ikeda's construction in \cite{Ikeda80_isosp-lens}. 
At the end, we show, explicitly in a table, the examples for small values of $k$.

For $k,n\in\mathbb N$ with $n\geq2$ and $k\geq3$, we set 
\begin{equation}\label{eq:GersonFamily}
\GG(n,k) = \left\{L(k;s) :  
\begin{array}{l}
	\gcd(k,s_i)=1\;\forall i, \\ 
	s_i\not\equiv s_j\pmod k\quad\forall i\neq j,
	\\
	%(\spadesuit)\quad 
	\exists\, i\; \text{ such that  } s_i\not\equiv -s_j\pmod k \quad \forall j
\end{array}
\right\}.
\end{equation}
Clearly, $\GG(n,k)$ is empty if $n$ is greater than the size of $\mathbb Z_{k}^\times:=(\Z/k\Z)^\times$. 
Moreover, one has $\,\GG(n,k) \neq \emptyset \iff n<\varphi(k).$

\begin{theorem}\label{thm:Gerson}
If $k$ is an odd prime number, then the CR lens spaces in $\GG(k-3,k)$ are mutually CR isospectral. 
\end{theorem}

\begin{proof}
Write $n=k-3$.
We want to show that $F_{L}(z,w)=F_{L'}(z,w)$ for all $L,L'\in\GG(n,k)$, which is sufficient by Proposition~\ref{prop:F_Gamma=>CR-isosp}.
The strategy is to show that $F_{L(q;s)}(z,w)$ does not depend on $s=(s_1,\dots,s_n)$ if $L(q;s)\in\GG(n,k)$. 

We fix $L:=L(k;s)\in\GG(n,k)$. 
Since $s_i\not\equiv s_j\pmod k$ for all $i\neq j$ and the number of elements in $\mathbb Z_{k}^\times:=(\Z/k\Z)^\times$ is $k-1=n+2$ because $k$ is prime, there are $u,v\in\Z$ prime to $k$ such that $\{s_1,\dots,s_n,u,v\}$ is a representative set of $\mathbb Z_{k}^\times$.
In other words, for any $a\in\Z$ prime to $k$, there is a unique element $t$ in this set such that $a\equiv t\pmod k$. 
Furthermore, $u\not\equiv -v\pmod k$ by the third condition in \eqref{eq:GersonFamily}.

Let us denote by $\Phi_l$ the $l$th cyclotomic polynomial. 
Notice $\Phi_k(z)=\prod_{i=1}^{n} (z-\xi_k^{i})$ since $k$ is a prime integer; recall the notation $\xi_k=e^{\frac{2\pi\mi}{k}}$. 
From \eqref{eq:F_L(z,w)}, we have that
\begin{equation}\label{eq:F_Lcyclotomic}
\begin{aligned}
F_{L}(z,w) &
= \frac{1-zw}{k} \left(
	\frac{1}{(z-1)^n(w-1)^n}
	+\frac{1}{\Phi_k(z)\Phi_k(w)} {\sum_{m=1}^{k-1} }
		\frac{\Phi_k(z)}{\prod_{i=1}^{n} (z-\xi_k^{-ms_i})} 
		\frac{\Phi_k(w)}{\prod_{i=1}^{n} (w-\xi_k^{ms_i})}
\right)
.
\end{aligned}
\end{equation}
Note that $\{\xi_k^{ms_i}:  1\leq i\leq n\}\cup \{\xi_k^{mu}, \xi_k^{mv}\}$ is the set of $k$-th roots of unity for any $m=1,\dots,k-1$, thus
\begin{multline}\label{eq:prod}
\sum_{m=1}^{k-1} 
\frac{\Phi_k(z)}{\prod_{i=1}^{n} (z-\xi_k^{-ms_i})} \frac{\Phi_k(w)}{\prod_{i=1}^{n} (w-\xi_k^{ms_i})} 
= \sum_{m=1}^{k-1}  (z-\xi_k^{-mu})(z-\xi_k^{-mv}) (w-\xi_k^{mu})(w-\xi_k^{mv})
\\ 
=\sum_{m=1}^{k-1} \Big(
z^2w^2
- zw^2 (\xi_k^{-mu}+\xi_k^{-mv})
- z^2w (\xi_k^{mu}+\xi_k^{mv})
+z^2\xi_k^{m(u+v)} +w^2\xi_k^{-m(u+v)}
\\
+zw\big(2+\xi_k^{-m(u-v)}+\xi_k^{m(u-v)}\big)
-z\big( \xi_k^{mu}+\xi_k^{mv}\big)
-w\big( \xi_k^{-mu}+\xi_k^{-mv}\big)
+1\Big)
\\
=(k-1)z^2w^2
+2 zw^2
+2 z^2w
-z^2 -w^2
+(2k-4)zw 
+2z
+2w
+(k-1)
.
\end{multline}
In the last row we have used the classical identity 
\begin{equation}
\sum_{m=0}^{l-1}\xi_l^{mt}=
\begin{cases}
l&\text{ if }l\text{ divides } t,\\
0&\text{ otherwise}.
\end{cases}
\end{equation}
In particular, $\sum_{m=1}^{k-1} \xi_k^{\pm m(u-v)}=-1$ since $k\nmid u-v$. 

Now, writing by $p(z)$ the resulting polynomial in \eqref{eq:prod} and combining it with \eqref{eq:F_Lcyclotomic}, we have that 
\begin{equation}
F_{L}(z,w)
= \frac{1-zw}{k} \left(
	\frac{1}{(z-1)^n(w-1)^n}
	+\frac{p(z)}{\Phi_k(z)\Phi_k(w)}
\right),
\end{equation}
which does not depend on $s_1,\dots,s_n$ as asserted. 
\end{proof}

For $k$ an odd prime number, there are several pairs of CR lens spaces in $\GG(k-3,k)$ that are indeed equivalent. 
The next goal is to show that the number of equivalence classes of CR lens spaces in $\GG(k-3,k)$ has a cardinality going to $\infty$ as $k\to\infty$.

\begin{proposition}\label{prop:arbitraryliylarge}
If $k$ is an odd prime number, then 
there are exactly $\frac{k-3}2$ equivalence classes of CR lens spaces in $\GG(k-3,k)$. 
Hence, there are $\frac{k-3}2$ mutually CR isospectral not equivalent CR lens spaces in $\GG(k-3,k)$. 
\end{proposition}

\begin{proof}
First we count the different members $L(k;s)$ in $\GG(k-3,k)$.
We have to choose $k-3$ different elements $s_i$ in $\mathbb{Z}_k^\times$ in such a way that the two unchosen elements
do not add up to $k \pmod k$, i.e., they are not opposite in $\mathbb{Z}_k^\times$. 
There are $\left(\binom{k-1}{2}-\frac{k-1}{2}\right)(k-3)!$ ways of doing this, since here we are taking into account the order of the elements $s_i$, and this is the number  $\#\GG(k-3,k)$.

Now, we note that if $L(k;s)$ belongs to $\GG(k-3,k)$, then either does the lens space obtained by multiplying the $s_i$'s by a scalar in $\mathbb{Z}_k^\times$, and then applying a permutation $\sigma$ to the indices.
Furthermore, all the lens spaces obtained in this way will be different to each other, since  
we can reduce this computation to the case when the two lens spaces are obtained from a fixed $L(k;s)$, the first one by permuting the elements $s_i$ with $\sigma\in S_{k-3}$, and the second one by multiplying all the $s_i$ by a fixed $c\in\mathbb{Z}_k^\times$. 
Now the first one has the same two unchosen elements as $L(k;s)$ had, so the same must happen with the second lens space. 
Hence multiplication by $c$ must take the two elements in $\mathbb{Z}_k^\times$ into the same set, 
and, by the assumption that these two elements are not opposite, it turns out that $c$ must be equal to $1$. 
This, in turn, implies that $\sigma$ is the identity. 

Thus, each class of equivalent CR lens spaces in $\GG(k-3,k)$
has exactly $(k-3)!(k-1)$ elements. Hence, the cardinality of $\GG(k-3,k)$, up to CR diffeomorphisms is
$$
\frac{(\binom{k-1}{2}-\frac{k-1}{2}) \, (k-3)!} {(k-3)!\, (k-1)}= \frac{k-3}2,
$$
as claimed.
\end{proof}

\begin{example}
The next table shows the CR isospectral family provided by Theorem~\ref{thm:Gerson}, up to equivalence, for small values of $k$. 
Note that for $k=5$, the family has a single element. 
\begin{equation}
\begin{array}{ccl}
k & n & \GG(n,k)/\text{equivalence} \\ \hline
\rule{0pt}{14pt}
5 & 2 & 
\left\{\begin{array}{l}
	L(5;1,2)
\end{array}\right\} 
\\
7 & 4 & 
\left\{\begin{array}{l}
	L(7;1,2,3,4), L(7;1,2,3,5)
\end{array}\right\} 
\\
11 & 8 & 
\left\{\begin{array}{l}
	L(11;1,2,3,4,5,6,7,8), L(11;1,2,3,4,5,6,7,9),\\ L(11;1,2,3,4,5,6,7,10), L(11;1,2,3,4,5,6,9,10)
\end{array}\right\} 
\\
13 & 10 & 
\left\{\begin{array}{l}
	L(13;1,2,3,4,5,6,7,8,9,10), 
	L(13;1,2,3,4,5,6,7,8,9,11),\\
	L(13;1,2,3,4,5,6,7,8,9,12),
	L(13;1,2,3,4,5,6,7,8,10,11), \\
	L(13;1,2,3,4,5,6,7,9,11,12)
\end{array}\right\} 
\end{array}
\end{equation}
\end{example}

\section{Finiteness condition and computational results} \label{sec:computational}
The main goal of this section is to find, with the help of a computer, CR isospectral pairs of $(2n-1)$-dimensional CR lens spaces with fundamental group of order $k$, for small values of $n$ and $k$. 
To do that, we first obtain a finiteness condition about whether two such CR lens spaces are CR isospectral.

Let $L=L(k;s)$ be a $(2n-1)$-dimensional CR lens space. 
By multiplying and dividing by $(z^k-1)^n(w^k-1)^n$ to $F_{L}(z,w)$, \eqref{eq:F_L(z,w)} gives 
\begin{equation}\label{eq:F_Lcomputacional}
F_{L}(z,w) 
=\frac{1}{k}\frac{1-zw}{(z^k-1)^n(w^k-1)^n} \sum_{m=0}^{k-1} 
\left(\prod_{i=1}^{n}
	\frac{z^k-1}{ z-\xi_k^{-s_im}}
	\frac{w^k-1}{w-\xi_k^{s_im}}
\right)
.
\end{equation}
On the one hand, the common factor depends only on $n$ and $k$, but not on $s=(s_1,\dots,s_n)$.
On the other hand, each term inside the sum is a polynomial on $z$ and $w$, and the whole sum is a polynomial of degree at most $n(k-1)$ in both variables that we denote by
\begin{equation}\label{eq:P_L(z,w)}
P_{L}(z,w) =\sum_{\alpha,\beta=0}^{n(k-1)} a_{\alpha,\beta}\, z^\alpha w^\beta, 
\quad \text{i.e.\ }
F_{L}(z,w) 
=\frac{1}{k}\frac{(1-zw)P_{L}(z,w)}{(z^k-1)^n(w^k-1)^n} .
\end{equation}   
Consequently, the spectrum of the Kohn Laplacian on $L(q;s)$ is determined by the finitely many coefficients $a_{\alpha,\beta}$ for $1\leq \alpha,\beta\leq n(k-1)$. 
We next show how to obtain the values of $\{a_{\alpha,\beta}\}_{\alpha,\beta}$ from finitely many coefficients of $F_{L}(z,w)$.

\begin{theorem}\label{thm:finitecondition}
Let $L=L(k;s)=S^{2n-1}/\Gamma$ and $L'=(k;s')=S^{2n-1}/\Gamma$ be lens spaces of dimension $2n-1$. 
If 
\begin{equation*}
\dim \HH_{p,q}^{\Gamma}=\dim \HH_{p,q}^{\Gamma'}
\qquad\text{for all \ $0\le p,q\le n(k-1)$}
,
\end{equation*}
then $F_{L}(z,w)=F_{L'}(z,w)$ and consequently the CR lens spaces $L$ and $L'$ are CR isospectral. 
\end{theorem}

\begin{proof} 
The strategy of the proof is to expand both sides of 
\begin{equation}\label{eq:LHS-RHS}
k(z^k-1)^n(w^k-1)^n F_L(z,w) =(1-zw)P_L(z,w).
\end{equation} 
On the one hand, by expanding the two binomials, the left-hand side of \eqref{eq:LHS-RHS} equals
\begin{multline*}
k\sum_{p,q\geq 0} \sum_{i=0}^n \sum_{j=0}^n  (-1)^i(-1)^j  \binom {n}{i} \binom {n}{j} \dim\HH_{p,q}^\Gamma z^{p+ ik} w^{q+ jk} 
\\
= k\sum_{\alpha,\beta\geq0}
	\sum_{i=0}^{\lfloor \alpha/k\rfloor}\sum_{j=0}^{\lfloor \beta/k\rfloor}
	(-1)^{i+j} \binom {n}{i} \binom {n}{j} \dim\HH_{\alpha-ik,\beta-jk}^\Gamma
	\; z^\alpha w^\beta. 
\end{multline*}
On the other hand, standard computations show that the right-hand side of \eqref{eq:LHS-RHS} equals
\begin{multline*}
\sum_{\alpha,\beta \geq0}^{n(k-1)} a_{\alpha,\beta} z^\alpha w^\beta - \sum_{\alpha,\beta \geq0}^{n(k-1)} a_{\alpha,\beta} z^{\alpha +1} w^{\beta +1} 
\\
= a_{0,0} 
+ \sum_{\beta \geq1}^{n(k-1)} a_{0,\beta} w^\beta 
+ \sum_{\alpha\geq1}^{n(k-1)} a_{\alpha,0} z^\alpha 
+\sum_{\alpha,\beta \geq1}^{n(k-1)} \big( a_{\alpha,\beta} - a_{\alpha-1,\beta-1} \big) z^\alpha w^\beta 
+\text{other terms}
,
\end{multline*}
where `other terms' belongs to the span of $z^{\alpha}w^{\beta}$ with $\alpha>n(k-1)$ or $\beta>n(k-1)$.

Combining the expressions on both sides, it is clear that we can obtain the coefficients of the polynomial $P_{L}(z,w)$ in terms of $\{\dim \HH_{p,q}^\Gamma: 0\leq p,q,\leq n(k-1)\}$. 
Indeed, for any $0\leq \alpha,\beta\leq n(k-1)$, 
\begin{equation}
a_{\alpha,\beta} =
\begin{cases}
k \sum_{i=0}^{\lfloor \alpha/k\rfloor}	(-1)^{i} \binom {n}{i} \dim\HH_{\alpha-ik,0}^\Gamma 
	&\text{if }\beta=0,\\
k \sum_{j=0}^{\lfloor \beta/k\rfloor}	(-1)^{j} \binom {n}{j} \dim\HH_{0,\beta-jk}^\Gamma 
	&\text{if }\alpha=0,\\
a_{\alpha-1,\beta-1} 
+k \sum_{i=0}^{\lfloor \alpha/k\rfloor}\sum_{j=0}^{\lfloor \beta/k\rfloor}
	(-1)^{i+j} \binom {n}{i} \binom {n}{j} \dim\HH_{\alpha-ik,\beta-jk}^\Gamma
	&\text{if }\alpha\beta>0,
\end{cases}
\end{equation}
and an adequate recursion gives all the values of $a_{\alpha,\beta}$. 

We have shown so far that the values $\{\dim\HH_{p,q}: 0\leq p,q\leq n(k-1)\}$ determine $P_{L}(z,w)$, and consequently also $F_{L}(z,w)$ by \eqref{eq:P_L(z,w)}. 
We conclude that $\dim \HH_{p,q}^{\Gamma}=\dim \HH_{p,q}^{\Gamma'}$ for all $0\le p,q\le n(k-1)+1$ implies $F_{L}(z,w)=F_{L'}(z,w)$ and consequently the CR lens spaces $L$ and $L'$ are CR isospectral.
\end{proof}

\begin{remark}
Theorem~\ref{thm:finitecondition} does not answer the interesting question of whether \emph{a finite part of the spectrum of the Kohn Laplacian $\square_{b,L}$ on the CR lens space $L$ determines the whole spectrum of $\square_{b,L}$}.  
See \cite[Cor.5.2]{LauretLinowitz} for a related result valid for spherical space forms.
\end{remark}

This finiteness condition in Theorem~\ref{thm:finitecondition} allows us, with the help of a computer, to determine all the
pairs $L,L'$ of non-equivalent 
$(2n-1)$-dimensional CR lens spaces with fundamental group of order $k$ satisfying $F_L(z,w)=F_{L'}(z,w)$ (so $L$ and $L'$ are CR isospectral), at least for small values of $n$ and $k$. 

Table~\ref{table:isosp} shows all CR isospectral families for some low values of $n$ and $k$. 
It turned out that there is no any pair of $3$-dimensional (non-equivalent) CR isospectral CR lens spaces with fundamental groups of order $\leq 500$. 

\begin{problem}
Prove (or disprove) that two CR isospectral CR lens spaces of dimension $3$ are equivalent. 
\end{problem}

Fan, Kim, Plzak, Shors, Sottile and Zeytuncu~\cite[Thm.~2.4]{Yunus-cia} solved the above problem when the order of the fundamental $k$ is a prime number. 
In the Riemannian setting, Ikeda and Yamamoto~\cite{IkedaYamamoto79} established the analogous result, and later Yamamoto~\cite{Yamamoto80} extended it to an arbitrary $k$. 

\begin{table}
$
\begin{array}[t]{ccl}
2n-1 & k & \text{family} \\ \hline \rule{0pt}{14pt}
5 & 49 & [1, 8, 22], [1, 8, 36] \\
5 & 64 & [1, 9, 25], [1, 9, 49] \\
5 & 98 & [1, 15, 43], [1, 15, 71] \\
5 & 100 & [1, 11, 31], [1, 11, 81] \\
5 & 100 & [1, 11, 41], [1, 11, 71] \\
5 & 121 & [1, 12, 34], [1, 12, 100] \\
5 & 121 & [1, 12, 45], [1, 12, 89] \\
5 & 121 & [1, 12, 56], [1, 12, 78] \\
5 & 121 & [1, 23, 56], [1, 23, 89] \\
5 & 121 & [1, 23, 67], [1, 23, 78] \\
7 & 7 & [1, 2, 3, 4], [1, 2, 3, 5] \\
7 & 14 & [1, 3, 5, 9], [1, 3, 5, 13] \\
7 & 21 & [1, 4, 10, 16], [1, 4, 10, 19] \\
7 & 28 & [1, 5, 9, 17], [1, 5, 9, 25] \\
7 & 32 & [1, 5, 9, 13], [1, 5, 9, 29] \\
7 & 35 & [1, 6, 11, 16], [1, 6, 11, 26] \\
7 & 42 & [1, 13, 19, 25], [1, 13, 19, 31] \\
7 & 49 & [1, 8, 15, 29], [1, 8, 15, 36] \\
7 & 56 & [1, 9, 17, 25], [1, 9, 17, 33] \\
7 & 63 & [1, 10, 19, 37], [1, 10, 19, 55] \\
7 & 70 & [1, 11, 31, 51], [1, 11, 31, 61] \\
7 & 72 & [1, 7, 13, 31], [1, 7, 13, 55] \\
7 & 72 & [1, 7, 25, 67], [1, 7, 31, 61] \\
7 & 77 & [1, 12, 23, 45], [1, 12, 23, 67] \\
7 & 81 & [1, 10, 19, 37], [1, 10, 19, 64] \\
7 & 81 & [1, 10, 19, 46], [1, 10, 19, 55]  \\
7 & 81 & [1, 10, 28, 46], [1, 10, 46, 64]  \\
7 & 84 & [1, 13, 25, 37], [1, 13, 25, 61]  \\
7 & 91 & [1, 27, 40, 53], [1, 27, 40, 66]  \\
7 & 96 & [1, 13, 25, 37], [1, 13, 25, 85]  \\
7 & 98 & [1, 15, 29, 57], [1, 15, 29, 71]  \\
7 & 100 & [1, 11, 21, 41], [1, 11, 21, 81] \\
7 & 100 & [1, 11, 31, 71], [1, 11, 41, 81] \\
\end{array}
\quad 
\begin{array}[t]{ccl}
2n-1 & k & \text{family} \\ \hline \rule{0pt}{14pt}
9 & 16 & [1, 1, 3, 3, 9], [1, 1, 3, 3, 11] \\
9 & 16 & [1, 1, 5, 5, 9], [1, 1, 5, 5, 13] \\
9 & 16 & [1, 3, 5, 7, 9], [1, 3, 5, 7, 11] \\
9 & 16 & [1, 3, 5, 11, 15], [1, 3, 5, 13, 15] \\
9 & 20 & [1, 3, 7, 13, 19], [1, 3, 7, 17, 19] \\
9 & 27 & [1, 4, 7, 10, 16], [1, 4, 7, 10, 22] \\
9 & 27 & [1, 4, 7, 13, 19], [1, 4, 10, 22, 25] \\
9 & 27 & [1, 4, 7, 13, 25], [1, 4, 7, 19, 25] \\
9 & 32 & [1, 1, 7, 7, 17], [1, 1, 7, 7, 23]  \\
9 & 32 & [1, 1, 9, 9, 17], [1, 1, 9, 9, 25]  \\
9 & 32 & [1, 3, 9, 11, 17], [1, 3, 9, 11, 19] \\
9 & 32 & [1, 3, 9, 17, 19], [1, 3, 9, 19, 25] \\
9 & 32 & [1, 3, 5, 17, 25], [1, 3, 7, 11, 27]  \\
9 & 32 & [1, 5, 9, 13, 17], [1, 5, 9, 13, 21], \\
&  & [1, 5, 9, 13, 29], [1, 5, 9, 17, 29] \\
9 & 32 & [1, 5, 9, 17, 21], [1, 5, 9, 21, 25] \\
9 & 32 & [1, 7, 9, 15, 17], [1, 7, 9, 15, 23] \\
9 & 32 & [1, 7, 9, 23, 31], [1, 7, 9, 25, 31] \\
11 & 11 & [1, 2, 3, 4, 7, 9], [1, 2, 3, 5, 6, 10] \\
11 & 11 & [1, 2, 3, 5, 7, 8], [1, 2, 3, 4, 5, 9] \\
11 & 11 & [1, 2, 3, 4, 6, 7], [1, 2, 3, 4, 6, 10] \\
11 & 11 & [1, 2, 3, 4, 5, 8], [1, 2, 3, 4, 5, 10] \\
11 & 19 & [1, 4, 5, 6, 7, 11], [1, 4, 5, 6, 7, 17] \\
11 & 22 & [1, 3, 5, 7, 9, 13], [1, 3, 5, 7, 9, 17] \\
11 & 22 & [1, 3, 5, 7, 9, 15], [1, 3, 5, 7, 13, 19] \\
11 & 22 & [1, 3, 5, 7, 13, 17], [1, 3, 7, 9, 13, 17] \\
11 & 22 & [1, 3, 5, 7, 17, 19], [1, 3, 5, 9, 17, 21] \\
11 & 25 & [1, 1, 6, 6, 11, 16], [1, 1, 6, 6, 16, 21] \\
11 & 25 & [1, 1, 6, 11, 11, 16], [1, 1, 6, 11, 11, 21] \\
11 & 25 & [1, 4, 6, 9, 11, 16], [1, 4, 6, 9, 11, 21] \\
11 & 25 & [1, 4, 6, 11, 14, 16], [1, 4, 6, 11, 16, 19] \\
11 & 25 & [1, 4, 6, 11, 14, 19], [1, 4, 6, 14, 21, 24] \\
11 & 25 & [1, 4, 6, 9, 16, 21], [1, 4, 6, 9, 19, 24] \\
11 & 25 & [1, 4, 6, 9, 16, 24], [1, 4, 6, 9, 19, 21] \\
11 & 25 & [1, 4, 6, 11, 19, 21], [1, 4, 6, 16, 19, 21] \\
\end{array}
$

\bigskip 

\caption{
	Families of $(2n-1)$-dimensional CR isospectral CR lens spaces with fundamental group of order $k$, with $k$ at most $121,100,32$ and $25$ for $n=3,4,5$ and $6$ respectively.  
	Note that for $2n-1=9$ and $k=32$ there are four CR lens spaces which are mutually CR isospectral. 
}
\label{table:isosp}
\end{table}

Curiously, the $5$-dimensional CR lens spaces occurring in Table~\ref{table:isosp} are indeed $p$-isospectral for all $p$ as Riemannian lens spaces (cf.\ \cite{LMR-onenorm}). 
For instance, as Riemannian lens spaces we have the isometries $L(49;1, 8, 22)\simeq L(49;1, 6, 15)$ and $L(49;1, 8, 36)\simeq L(49;1, 6, 20)$ by Proposition~\ref{prop:isometria}, while $L(49;1, 6, 15), L(49;1, 6, 20)$ is the pair of highest volume (i.e.\ smallest fundamental group) of $5$-dimensional lens spaces $p$-isospectral for all $p$ found in \cite[Table~1]{LMR-onenorm}. 
Computational experiments for $k\leq 300$ provide evidence that there is a bijection between pairs of $5$-dimensional (non-isometric) Riemannian lens spaces $p$-isospectral for all $p$ with fundamental group of order $k$ and pairs of $5$-dimensional CR isospectral (non-equivalent) CR lens spaces with fundamental group of order $k$. 

\begin{question}
What is the exact relation between $5$-dimensional Riemannian lens spaces $p$-isospectral for all $p$ and $5$-dimensional CR isospectral CR lens spaces?
\end{question}

\section{A large family of CR isospectral pairs} \label{sec:CRisosppairs}

The main result of this section, Theorem~\ref{thm:theoremn}, provides an enormous family of pairs of CR isospectral CR lens spaces. 
This family is inspired by the families considered in \cite{LMR-onenorm,DD18}.

\begin{theorem}\label{thm:theoremn}
Let $n,r$ be integers such that $n\ge 3$ and $r>3$, $r$ odd, and we set 
$k=r^2$, $\tt=r+1$.
For integers $0\le a_1<a_2<\dots<a_n<r$ such that $\gcd(a_i-a_j,r)=1$ for all $i\neq j$,
the CR lens spaces
\begin{equation}
\begin{aligned}
L^+&:= L\big(r^2; \tt^{a_1}, \tt^{a_2},\dots, \tt^{a_n}\big),\\
L^-&:= L\big(r^2; \tt^{-a_1}, \tt^{-a_2},\dots, \tt^{-a_n}\big),
\end{aligned}
\end{equation} 
are CR isospectral.
Moreover, $F_{L^+}(z,w)=F_{L^-}(z,w)$.  
\end{theorem}

\begin{proof}
From \eqref{eq:F_Lcomputacional}, we have to show that 
\begin{equation}\label{eq:identitypolynomials}
\sum_{m=0}^{k-1} \frac{(z^k-1)(w^k-1)}{\prod_{j=1}^n(z-\xi^{m\tt^{a_j}}) (w-\xi^{-m\tt^{a_j}})} 
= \sum_{m=0}^{k-1} \frac{(z^k-1)(w^k-1)}{\prod_{j=1}^n(z-\xi^{m\tt^{-a_j}}) (w-\xi^{-m\tt^{-a_j}})}
.
\end{equation}
For simplicity, we have written $\xi=\xi_k=e^{2\pi\mi/k}$.
Note that $\theta^{a}\equiv ar+1\pmod{r^2}$ for all $a\in\Z$. 
It follows that the $m$-th terms in \eqref{eq:identitypolynomials} on both sides coincide if $r\mid m$, since if we look at the denominators when $m=ar$ we have 
$\,
z-\xi^{m\theta^{\pm a_j}} = z-\xi^{ar(1\pm a_jr)} 
= z-\xi^{ar}$,
for $j=1,2,\dots, n$, and similarly for the variable $w$. 
Thus, for $a=0,1,\dots,r-1$, the $m$-th term for both, when $r\mid m=ar$, is
$$
\frac{(z^k-1)(w^k-1)}{(z-\xi^{ar})^n(w-\xi^{-ar})^n}
$$
hence it can be simplified.
Moreover, 
\begin{equation*}
P_{L^\pm}(z,w):= 
\sum_{\substack{1\leq m\leq k-1,\\ r\,\nmid \, m}} \frac{(z^k-1)(w^k-1)}{\prod_{j=1}^n(z-\xi^{m\tt^{\pm a_j}})(w-\xi^{-m\tt^{\pm a_j}})}
\end{equation*}
are polynomials in $z$ and $w$ since the $k$-th roots of unity $\{\xi^{m\theta^{\pm a_j}} : 1\leq j\leq n\}$ (resp.\ $\{\xi^{-m\theta^{\pm a_j}} : 1\leq j\leq n\}$) are pairwise different. 
Indeed, $m\theta^{\pm a_i}\equiv m\theta^{\pm a_j}\pmod {r^2}$ if and only if $\pm m(a_i-a_j)\equiv 0\pmod {r}$, which is not possible unless $i=j$ because $r\nmid m$ and $\gcd(a_i-a_j,r)=1$ if $i\neq j$.

We now fix $w$ and 
consider $P_{L^\pm}(z,w)$ as polynomials in the variable $z$; they are of degree $r^2-n$. 
We claim that $P_{L^+}(\xi^{m_0},w)=P_{L^-}(\xi^{m_0},w)$ for all ${m_0}\in\Z$, $0\le m_0\le r^2-1$, which will imply that they are equal.
First, if $m_0=ar$,  $a=0,\dots,r-1$, we notice that both polynomials are equal to zero, since the root $\xi^{ar}$ occurs in the numerators and not in the denominators.  In the case $m_0\neq ar$,
clearly, the $m$-th term of $P_{L^{\pm}}(\xi^{m_0},w)$ vanishes unless $m=m_0\tt^{\mp a_i}$ for some $1\leq i\leq n$. 
Thus, 
\begin{equation*}
P_{L^\pm}(z_0,w)
= \sum_{i=1}^n  \frac{(z_0^k-1)(w^k-1)}{\prod_{j=1}^n(z_0-\xi^{m_0\tt^{\pm( a_j-a_i)}})(w-\xi^{-m_0\tt^{\pm( a_j-a_i)}})}
,
\end{equation*}
where $z_0=\xi^{m_0}$. 
We note that in this expression all the factors in the denominators occurs also in the numerators and 
could be simplified.

Now, there will be a simplification of factors when comparing the $P_{L^\pm}(z_0,w)$. 
We define the set $E=\{ m_0 \theta^{a_j-a_i} : i \neq j \}$, and $|E|$ is the cardinal of $E$.   
We note that
$$(z_0^k-1)(w^k-1)=\prod_{\ell\in E} (z_0-\xi^\ell)(w-\xi^{-\ell}) 
\prod_{\ell\not\in E} (z_0-\xi^\ell)(w-\xi^{-\ell}).$$
As said above, in each expression, the factors in the denominators when $j=i$ cancel out with the corresponding factors in the numerators (those with $\ell=m_0$). 
And since the product $\prod_{\ell\not\in E,\,\ell\ne m_0} (z_0-\xi^\ell)(w-\xi^{-\ell})$ occurs in the numerator of every term and never in the denominator, it can be simplified. 
Hence, by setting  
\begin{equation*}
\tilde P_{L^\pm}(z_0,w)= \sum_{i=1}^{n} \frac{\prod_{\ell\in E} (z_0-\xi^{\ell})(w-\xi^{-\ell})}{\prod_{j=1,j\neq i}^{n}(z_0-\xi^{m_0\tt^{\pm(a_j-a_i)}})(w-\xi^{-m_0\tt^{\pm(a_j-a_i)}})},
\end{equation*}
it suffices to prove $\tilde P_{L^+}(z_0,w)=\tilde P_{L^-}(z_0,w)$. 
We will prove this for every $w$ by considering
it as a polynomial equality. Notice that both are polynomials of degree $|E|-(n-1)$ in the variable $w$, and we will prove that they are equal for $|E|$ different values of $w$, namely,  $w_0:=\xi^{-\ell_0}$, $\ell_0\in E$.  

We first note that most of the terms in $\tilde P_{L^\pm}(z_0,w_0)$  vanish. 
However, it does not vanish in $\tilde P_{L^+}(z_0,w_0)$ the term with $i=s$ when there is a $t$, $1\le t\le n$, such that 
$\xi^{-m_0 \tt^{(a_t-a_s)}}=\xi^{-\ell_0}$.
In this case, it does not vanish in $\tilde P_{L^-}(z_0,w_0)$ the term with $i=t$. Moreover,
we will prove that these two terms are equal in the following way. Since in these terms we have the factor $(w_0-\xi^{-\ell_0})$ in the numerator and in the denominator, then it is convenient to ignore them, hence we can write the identity that we have to prove as follows
\begin{align*}
    \frac{\prod_{\ell\in E-\{\ell_0\}} (z_0-\xi^{\ell})(w_0-\xi^{-\ell})}{\prod_{j=1,j\neq s,t}^{n}(z_0-\xi^{m_0\tt^{a_j-a_s}})(w_0-\xi^{-m_0\tt^{a_j-a_s}})} &= \frac{\prod_{\ell\in E-\{\ell_0\}} (z_0-\xi^{\ell})(w_0-\xi^{-\ell})}{\prod_{j=1,j\neq t,s}^{n}(z_0-\xi^{m_0\tt^{a_t-a_j}})(w_0-\xi^{-m_0\tt^{a_t-a_j}})}.
\end{align*}
Since all factors are different from zero and the numerators are equal, they can be simplified, 
and we are done by verifying the following identity
$$
\prod_{j=1,j\neq t,s}^{n}(z_0-\xi^{m_0\tt^{a_t-a_j}})(w_0-\xi^{-m_0\tt^{a_t-a_j}}) =
\prod_{j=1,j\neq s,t}^{n}(z_0-\xi^{m_0\tt^{a_j-a_s}})(w_0-\xi^{-m_0\tt^{a_j-a_s}}).  
$$
In order to check this, we will show that for each $j$, the factors corresponding to the index $j$ are the same in both sides.  
\begin{align*}
    (z_0-\xi^{m_0\tt^{a_t-a_j}})(w_0-\xi^{-m_0\tt^{a_t-a_j}}) &= (z_0-\xi^{m_0\tt^{a_j-a_s}})(w_0-\xi^{-m_0\tt^{a_j-a_s}}),  \quad \textrm{or equivalently}     \\
    z_0w_0-z_0\xi^{-m_0\tt^{a_t-a_j}}-w_0\xi^{m_0\tt^{a_t-a_j}} +1 &= z_0w_0-z_0\xi^{-m_0\tt^{a_j-a_s}}-w_0\xi^{m_0\tt^{a_j-a_s}} +1. 
\end{align*}
After simplifications and replacements ($z_0=\xi^{m_0}$, $w_0=\xi^{-m_0\tt^{a_t-a_s}}$), we have  
\begin{align*}
    \xi^{m_0}\xi^{-m_0\tt^{a_t-a_j}} + \xi^{-m_0\tt^{a_t-a_s}}\xi^{m_0\tt^{a_t-a_j}} &= \xi^{m_0}\xi^{-m_0\tt^{a_j-a_s}} + \xi^{-m_0\tt^{a_t-a_s}}\xi^{m_0\tt^{a_j-a_s}},  \quad \textrm{or equivalently} \\
    \xi^{m_0(a_j-a_t)r} + \xi^{m_0(a_s-a_j)r} &= \xi^{m_0(a_s-a_j)r} + \xi^{m_0(a_j-a_t)r}, 
\end{align*}
which is true, proving the identity claimed.

Now, if there is another pair $(s',t')\ne (s,t)$ such that   $\xi^{-m_0 \tt^{(a_{t'}-a_{s'})}}=\xi^{-\ell_0}$,
then $s'\ne s$, $t'\ne t$ and the term with $i=s'$ in $\tilde P_{L^+}(z_0,w_0)$ does not vanish and it is equal
to the term with $i=t'$ in $\tilde P_{L^-}(z_0,w_0)$, as we have already proved. The proof is complete since we took into account all the terms in the identity of $\tilde P_{L^+}(z_0,w_0)=\tilde P_{L^-}(z_0,w_0)$.
\end{proof}

The next result ensures that $L^+$ and $L^-$ as in Theorem \ref{thm:theoremn} are generically not equivalent as CR lens spaces. 

\begin{proposition}\label{prop:notequiv}
The CR lens spaces $L^+$ and $L^-$ as in Theorem \ref{thm:theoremn} are equivalent if and only if there is a permutation $\sigma \in  S_n$ which is the product of $\lfloor \frac{n}{2} \rfloor$ disjoint simple transpositions, and $a_i+a_{\sigma(i)} \equiv a_j + a_{\sigma(j)} {\pmod r}$, for all $i,j$. 
\end{proposition}
\begin{proof}
By Proposition \ref{prop:CR-isometria}, $L^+$ and $L^-$ are equivalent as CR lens spaces if and only if
there is a permutation $\sigma\in S_n$ and $c\in\Z$, prime to $r^2$, such that $c\,\tt^{\,a_{\sigma(i)}}\equiv \tt^{-a_i} \pmod {r^2}$, for all $i=1,\dots,n$, which is equivalent to 
\begin{equation}\label{eq:c}
c (1+a_{\sigma(i)}r) \equiv 1-a_ir \pmod {r^2}.
\end{equation}
By multiplying by $r$ we get $\,rc \equiv r  \pmod {r^2}$,
which implies that $c \equiv 1 \pmod r$.
Thus, we can write $\,c=1+\ell r$, and we have, for all $i=1,\dots,n$,
$$
(1+\ell r) (1+a_{\sigma(i)}r) \equiv 1-a_ir \pmod {r^2}. 
$$
Or equivalently, 
$\;r(\ell +a_i+a_{\sigma(i)}) \equiv 0 \pmod {r^2}$.
Then we have the following:
$$
\ell \equiv -(a_1+a_{\sigma(1)}) \equiv \dots \equiv -(a_n+a_{\sigma(n)}) \pmod r.
$$
In order to see the structure of the permutation $\sigma$, the identity 
$a_i + a_{\sigma(i)} \equiv a_{\sigma(i)} + a_{\sigma^2(i)} \pmod r$ implies that 
$a_i - a_{\sigma^2(i)} \equiv 0 \pmod r$,  
thus, by the assumptions, $\sigma^2(i)=i$, for all $i=1,\dots,n$.
Furthermore, $\sigma$ may have at most one fixed point, since in case $i$ and $j$ were fixed by $\sigma$, 
then we would have $-2a_i\equiv \ell \equiv -2a_j \pmod r$, whence $r$ divides $a_i-a_j$, a contradiction.
Hence, $\sigma$ is a product of $\lfloor \frac{n}{2} \rfloor$ disjoint simple transpositions.  

The proof of the converse assertion is direct, since if we set 
$\ell = -(a_i+a_{\sigma(i)})$ and $\,c=1+\ell r$, it is clear that 
\eqref{eq:c} holds, which is the algebraic condition for $L^+$ and $L^-$ to be equivalent. 
\end{proof}

\begin{corollary}
For each $n\ge 3$, there are infinitely many pairs of $(2n-1)$-dimensional CR lens spaces which are CR isospectral but not equivalent.
\end{corollary}
\begin{proof}
For $n$ fixed, we consider for each $r$ prime, $r\ge 2^n$,  a specific pair of lens spaces from the family in Theorem \ref{thm:theoremn}, namely, 
the parameters $a_i$ are $a_i:=2^{i-1}-1$, for $i=1, 2, \dots, n$, i.e.:
 $$L^+:=L\big( \theta-1)^2;\theta^0, \theta^1, \theta^3,\theta^7,\dots, \theta^{2^{n-1}-1}\big)$$
and, correspondingly, $L^-$ has the parameters $\theta^{-a_i}$ instead of $\theta^{a_i}$. 

First we verify that this pair belongs to the family. 
Since $r$ is prime and $r>a_n>\dots>a_1=0$, then  $(a_i-a_j,r)=1$ for every $i\neq j$. 
Hence, $L^+$ and $L^-$ are CR isospectral. 

Now, the condition in Proposition~\ref{prop:notequiv} for being equivalent as CR lens spaces 
is $\,a_i+a_{\sigma(i)}\equiv a_j+a_{\sigma(j)} \pmod r$, for all $i,j=1,\dots,n$, which since $r\ge 2^n$ reduces to,
$\,2^i+2^{\sigma(i)}=2^j+2^{\sigma(j)}$, which is possible only when $\{i,\sigma(i)\}=\{j,\sigma(j)\}$, for all $i,j=1,\dots,n$, 
which is clearly not possible to happen for $n\ge 3$. Hence $L^+$ and $L^-$ are not equivalent. 
\end{proof}{\scriptsize}

\begin{remark}
We just point out that most of the examples of CR isospectrality found with the help of a computer program as in Section~\ref{sec:computational}, having the order of the fundamental group $k=r^2$ are 
pairs already provided by the family above. It would be interesting to estimate the percentage of these examples
covered by this family. 
\end{remark}

\section{CR isospectral elliptic manifolds with non-cyclic fundamental groups} \label{sec:isospsphericalspaceforms}

The main goal in this section is to find finite subgroups $\Gamma,\Gamma'$ of $\Ut(d)$ acting freely on $S^{2d-1}$ such that $F_{\Gamma}(z,w)=F_{\Gamma'}(z,w)$, which implies that the elliptic CR manifolds $S^{2d-1}/\Gamma$ and $S^{2d-1}/\Gamma'$ are CR isospectral by Proposition~\ref{prop:F_Gamma=>CR-isosp}. 
In this section we have replaced the letter $n$ by $d$ for historical reasons.

Two finite subgroups $H,H'$ of a group $G$ are said to be \emph{almost conjugate in $G$} if there is a bijection between $H$ and $H'$ that preserves the conjugacy class in $G$. 

\begin{theorem}\label{thm:almostconjugate}
Let $\Gamma$ and $\Gamma'$ be two finite subgroups of $\Ut(d)$ acting freely on $S^{2d-1}$. 
If $\Gamma$ and $\Gamma'$ are almost conjugate in $\Ut(d)$, then $F_{\Gamma}(z,w)=F_{\Gamma'}(z,w)$, and consequently $S^{2d-1}/\Gamma$ and $S^{2d-1}/\Gamma'$ are CR isospectral. 
\end{theorem} 
\begin{proof} 
It is well known that two elements $\gamma,\gamma'$ belong to the same conjugacy class in $\Ut(d)$ if and only if their spectra coincide, which is equivalent to $\det(z-\gamma)=\det(z-\gamma')$.
Thus, by hypothesis, there is a bijection $\tau:\Gamma\to \Gamma'$ such that $\det(z-\gamma)=\det(z-\tau(\gamma))$ for all $\gamma\in\Gamma$. 

Now, Theorem~\ref{thm:F_Gamma(z,w)} implies that 
\begin{equation}
\begin{aligned}
F_{\Gamma}(z,w) &
= \frac1{|\Gamma|} \sum_{\gamma\in\Gamma} \frac{1-zw}{\det(z-\gamma) \det(w-\bar \gamma)}
= \frac1{|\Gamma|} \sum_{\gamma\in\Gamma} \frac{1-zw}{\det(z-\tau(\gamma)) \det(w-\overline{\tau(\gamma)})}
\\ &
= \frac1{|\Gamma'|} \sum_{\gamma'\in\Gamma'} \frac{1-zw}{\det(z-\gamma') \det(w-\overline{\gamma'})}
= F_{\Gamma'}(z,w),
\end{aligned}
\end{equation} 
as asserted. 
\end{proof}
 
The analogous result in the Riemannian context (i.e.\ for spherical space forms and isospectrality with respect to the Laplace-Beltrami operator on such Riemannian manifolds) was given by Ikeda in \cite[Cor.~2.3]{Ikeda80_3-dimI}.
Curiously, Ikeda's result is a particular case of the famous Sunada method \cite{Sunada85}, which had not been developed at that time. 

\begin{remark}
Under the same hypothesis as in Theorem~\ref{thm:almostconjugate}, the spherical space forms $S^{2d-1}/\Gamma$ and $S^{2d-1}/\Gamma'$ are \emph{strongly isospectral}, which means they are isospectral with respect to every natural strongly elliptic operator acting on sections of a natural vector bundle, including the Hodge-Laplace operator acting on $p$-forms. 
\end{remark}

The rest of the section is devoted to describe some known examples of almost conjugate subgroups of $\Ut(d)$ acting freely on $S^{2d-1}$ such that their corresponding elliptic manifolds are not diffeomorphic. 

For $m,n,r\in\N$ such that $\gcd(n(r-1),m)=1$ and $p\mid n'$ for all prime divisor $p$ of $d$, where $d$ is the order of $r$ in $\Z_n^\times$ and $n'=n/d$, we denote by $\Gamma_d(m,n,r)$ the group generated by $A,B$ with relations $A^m=e$, $B^n=e$, and $BAB^{-1} =A^r$.
These groups are called of Type I. 
They admit fixed point free complex representations, that is, $\rho:\Gamma_d(m,n,r)\to\Ut(d)$ such that the subgroup $\rho(\Gamma_d(m,n,r))$ of $\Ut(d)$ acts freely on $S^{2d-1}$, and consequently $S^{2d-1}/\rho(\Gamma_d(m,n,r))$ is an elliptic manifold.  
It is known that all irreducible fixed point free complex representations of $\Gamma_d(m,n,r)$ have degree $d$, thus the corresponding elliptic manifolds have dimension $2d-1$. 

\begin{theorem}\label{thm:Gamma_d(m,n,r)-irreps}
Let $\Gamma_d(m,n,r)$ be a group of Type I as above. 
If $\rho,\rho'$ are two irreducible fixed point free representations of $\Gamma_d(m,n,r)$, then $\rho(\Gamma_d(m,n,r))$ and $\rho'(\Gamma_d(m,n,r))$ are almost conjugate subgroups of $\Ot(2d)$. 
\end{theorem}

This result was first observed by Ikeda~\cite{Ikeda83}.
Later, Wolf~\cite{Wolf01} explained it as a particular case of a more general situation including all arbitrary fixed point free representations. 

It is standard that two irreducible representations as in Theorem~\ref{thm:Gamma_d(m,n,r)-irreps} provide isometric spherical space forms (see for instance \cite[Thm.~5.5.6]{Wolf-book} for a characterization of irreducible representations of Type I groups).
However, the next lemma by Ikeda (see \cite[Lem.~2.5]{Ikeda83}) ensures the existence of infinitely many pairs of isospectral spherical space forms with non-cyclic fundamental groups (see also \cite{Gilkey85}). 

\begin{lemma}[Ikeda~\cite{Ikeda83}]\label{lem:Ikeda}
Let $\Gamma= \Gamma_d(m,n,r)$ be a group of Type I as above such that $n=d^2$. 
The number of (Riemannian) isometry classes in $(2d-1)$-dimensional spherical space forms with the same fundamental group $\Gamma$ is at least $2$ if and only if $d=5$ or $d>6$. 
\end{lemma}

\begin{theorem}\label{thm:CRisosp-noncyclic}
For each positive integer $d$ with $d\geq5$ and $d\neq 6$, there are CR isospectral non-equivalent elliptic CR manifolds of dimension $2d-1$ with non-cyclic fundamental groups. 
\end{theorem}

\begin{proof}
We fix an integer $d$ as in the hypothesis. 
There is a Type I group of the form $\Gamma_d(m,n,r)$ with $n=d^2$ (see \cite[Lem.~2.6]{Ikeda83}). 
Lemma~\ref{lem:Ikeda} ensures that there are irreducible fixed point free representations $\rho,\rho'$ of $\Gamma$ such that the spherical space forms $S^{2d-1}/\rho(\Gamma_d(m,n,r))$ and $S^{2d-1}/\rho'(\Gamma_d(m,n,r))$ are non-isometric. 
On the one hand, Corollary~\ref{cor:isom=>CRequiv} forces that the elliptic CR manifolds $S^{2d-1}/\rho(\Gamma_d(m,n,r))$ and $S^{2d-1}/\rho'(\Gamma_d(m,n,r))$ are not equivalent.

On the other hand, Theorem~\ref{thm:Gamma_d(m,n,r)-irreps} implies that $\rho(\Gamma_d(m,n,r))$ and $\rho'(\Gamma_d(m,n,r))$ are almost conjugate in $\Ot(2d)$.
Note that this does not imply in general that $\rho(\Gamma_d(m,n,r))$ and $\rho'(\Gamma_d(m,n,r))$ are almost conjugate in $\Ut(d)$. 
However, the irreducible representations are of the form $\pi_{1,l}$ for some $l\in\Z$ with $\gcd(l,n)=1$ in the notation of \cite[Prop.~2.4]{Ikeda83} (see also \cite[Thm.~5.5.10]{Wolf-book}), and furthermore, the proof of \cite[Thm.~1]{Ikeda83} shows that $\Gamma=\pi_{1,l}(\Gamma)$ and $\Gamma'=\pi_{1,l'}(\Gamma)$ are in fact almost conjugate in $\Ut(d)$ (which does imply almost conjugacy in $\Ot(2n)$). 
We conclude that the elliptic CR manifolds $S^{2d-1}/\Gamma$ and $S^{2d-1}/\Gamma'$ are CR isospectral by Theorem~\ref{thm:almostconjugate}.  
\end{proof}

We next show some particular examples of CR isospectral elliptic CR manifolds with non-cyclic group provided by Theorem~\ref{thm:Gamma_d(m,n,r)-irreps}. 
Let $\Gamma= \Gamma_d(m,n,r)$ be a group of Type I. 
The fixed point free irreducible representations of $\Gamma_d(m,n,r)$ as above are equivalent to (see e.g.\ \cite[Thm.~5.5.6]{Wolf-book}) $\pi_{k,l}$ for some integers $k,l$ such that $\gcd(k,m)=1$, $\gcd(l,n)=1$, and $k\equiv 1\pmod d$, where
\begin{align}
\pi_{k,l}(A)&= 
\begin{pmatrix}
	\xi_m^{kr^0} \\
	& \xi_m^{kr^1} \\
	&&\ddots\\
	&&& \xi_m^{kr^{d-1}}
\end{pmatrix}
, 
&
\pi_{k,l}(B) &=
\begin{pmatrix}
0&1\\
\vdots &&\ddots \\
0&&&1\\
\xi_{n'}^{l}&0&\dots &0
\end{pmatrix}
.
\end{align}
As we mentioned in the proof of Theorem~\ref{thm:CRisosp-noncyclic}, we can assume that $k=1$ up to isometry of the corresponding spherical space forms, and furthermore, $\pi_{1,l}(\Gamma)$ and $\pi_{1,l'}(\Gamma)$ are almost conjugate in $\Ut(d)$ for all $l,l'\in\Z$ prime to $n$. 

Here the examples:
\begin{enumerate}
\item The group $\Gamma:=\Gamma_{5}(11,25,3)$ satisfies that $\pi_{1,1}(\Gamma)$ and $\pi_{1,2}(\Gamma)$ are almost conjugate in $\Ut(5)$ and $S^{9}/\pi_{1,1}(\Gamma)$ and $S^{9}/\pi_{1,2}(\Gamma)$ are non-isometric as Riemannian manifolds and non-equivalent as elliptic CR manifolds.

\item The group $\Gamma:=\Gamma_{7}(29,49,7)$ satisfies that $\pi_{1,1}(\Gamma)$, $\pi_{1,2}(\Gamma)$, and $\pi_{1,3}(\Gamma)$ are pairwise almost conjugate in $\Ut(7)$ and the corresponding $13$-dimensional elliptic CR manifolds are non-isometric as Riemannian manifolds and non-equivalent as elliptic CR manifolds. 

\item The group $\Gamma:=\Gamma_{13}(53,169,10)$ satisfies that $\pi_{1,l}(\Gamma)$ for $l=1,\dots,6$ are pairwise almost conjugate in $\Ut(13)$ and the corresponding $25$-dimensional elliptic CR manifolds are non-isometric as Riemannian manifolds and non-equivalent as elliptic CR manifolds. 
\end{enumerate}
Wolf~\cite{Wolf01} provided more examples of almost conjugate subgroups of $\Ut(n)$ of type II--VI.

We end the section with some open questions. 
The first asks whether the order of the fundamental group of an elliptic CR manifold is determined by the spectrum of the Kohn Laplacian. 
This is true in the Riemannian case; it follows immediately because the volume is a spectral invariant and $\vol(S^{2d-1}/\Gamma)= \vol(S^{2d-1})/|\Gamma|$;
see also \cite[Cor.~2.4]{Ikeda80_3-dimI} for another proof.

\begin{question}
Do CR isospectral spherical space forms have the same order of their fundamental groups? 
\end{question}

Fan, Kim, Plzak, Shors, Sottile and Zeytuncu~\cite[Cor.~2.3]{Yunus-cia} answered affirmatively this question for CR lens spaces.

To introduce the second question, we mention that the lens spaces $L(16; 1, 3, 5, 7, 9)$ and $L(16; 1, 3, 5, 7, 11)$ are diffeomorphic as differentiable manifolds (and isometric as Riemannian lens spaces) by Proposition~\ref{prop:isometria}, and furthermore, CR isospectral as CR lens spaces though they are not equivalent by Proposition~\ref{prop:CR-isometria}.
The next question asks whether this situation can occur for elliptic CR manifolds with non-cyclic fundamental groups. 

\begin{question}
Are there diffeomorphic elliptic manifolds with non-cyclic fundamental groups that are CR isospectral but not equivalent as CR manifolds?
\end{question}

An affirmative solution follows by showing irreducible representations $\pi_{k,l}$ and $\pi_{k',l'}$ of a Type I group $\Gamma_d(m,n,r)$ such that $\Gamma:=\pi_{k,l}(\Gamma_d(m,n,r))$ and $\Gamma':=\pi_{k',l'}(\Gamma_d(m,n,r))$ are conjugate in $\Ot(2d)$ but not in $\Ut(d)$. 
Indeed, the conjugacy in $\Ot(2d)$ between $\Gamma$ and $\Gamma'$ ensures that $S^{2n-1}/\Gamma$ and $S^{2n-1}/\Gamma'$ are isometric as spherical space forms by Theorem~\ref{thm:isometries} and therefore diffeomorphic, but the elliptic CR manifolds $S^{2n-1}/\Gamma$ and $S^{2n-1}/\Gamma'$ are not equivalent by Theorem~\ref{thm:CRequiv} for not being conjugate in $\Ut(d)$.

Ikeda~\cite[Thm.~3.9]{Ikeda80_3-dimII} proved that two isospectral $5$-dimensional spherical space forms with non-cyclic fundamental group are necessarily isometric.

\begin{question}
Are two CR isospectral $5$-dimensional elliptic CR manifolds with non-cyclic fundamental groups necessarily equivalent? 
\end{question}

Ikeda~\cite[Thm.~3.1]{Ikeda80_3-dimII} proved that, for $d$ an odd prime number, any pair of $(2d-1)$-dimensional isospectral spherical space forms $S^{2d-1}/\Gamma, S^{2d-1}/\Gamma'$ have isomorphic fundamental groups, that is, $\Gamma\simeq\Gamma'$. 
The next question asks if this extends to the CR context.

\begin{question}
Let $d$ be a prime number.
Are there non-isomorphic finite subgroups $\Gamma,\Gamma'$ of $\Ut(d)$ acting freely on $S^{2d-1}$ such that the CR manifolds $S^{2d-1}/\Gamma$ and $S^{2d-1}/\Gamma'$ are not CR equivalent and CR isospectral?
\end{question}

The main difficulty in responding affirmatively the above question is that CR-isospectral elliptic CR manifolds $S^{2d-1}/\Gamma, S^{2d-1}/\Gamma'$ with $\Gamma,\Gamma'\subset \Ut(n)$ do not satisfy necessarily that $F_{\Gamma}(z,w)=F_{\Gamma'}(z,w)$ in order to apply Theorem~\ref{thm:F_Gamma=>Riem-isosp} to obtain that the spherical space forms $S^{2d-1}/\Gamma$ and $S^{2d-1}/\Gamma'$ are isometric by Ikeda's result.

\section{Isospectrality of elliptic Riemannian manifolds endowed with Berger metrics}\label{sec:Berger}

In this section, as a byproduct of our previous results, we describe some isospectral examples with respect to the Laplace-Beltrami operator of elliptic Riemannian manifolds with certain non-round Riemannian metrics on the sphere. 

The group $\Ut(n)$ acts transitively on $S^{n-1}=\{(z_1,\dots,z_n)\in\C^n: |z_1|^2+\dots+|z_n|^2=1 \}$ by left multiplication. 
This gives the diffeomorphism $\Ut(n)/\Ut(n-1)\cong S^{2n-1}$ given by $a\Ut(n-1)\mapsto a\cdot e_{n}$. 
Although an analogous situation occurs by replacing $\Ut(n)$ by $\SU(n)$, we restrict our attention to $\Ut(n)$ for simplicity. 

A $\Ut(n)$-invariant metric $g$ on $S^{2n-1}$, which turns out to be Riemannian, is characterized by $g_{e_n}(X,Y)=g_{a\cdot e_n}\big((dL_a)_{e_n}(X), (dL_a)_{e_n}(Y)\big)$ for every $a\in \Ut(n)$ and $X,Y\in T_{e_n}S^{2n-1}$, where $L_a:S^{2n-1}\to S^{2n-1}$ is given by $L_a(p)=a\cdot p$ for $p\in S^{2n-1}$.

Besides the (homothetic) round metrics on $S^{2n-1}$, we have a two-parameter family of $\Ut(n)$-invariant metrics on $S^{2n-1}$ which are usually called \emph{Berger metrics}, that we next describe following \cite[\S3]{BLPfullspec}. 

We set $G=\Ut(n)$, $H=\{\left(\begin{smallmatrix} A&0\\ 0&1 \end{smallmatrix}\right) :  A\in\Ut(n-1)\}\simeq \Ut(n-1)$, and $K=\{\left(\begin{smallmatrix} A&0\\ 0&z \end{smallmatrix}\right) :  A\in\Ut(n-1), \, z\in\Ut(1) \}\simeq \Ut(n-1)\times \Ut(1)$. 
Note that $H\subset K\subset G$, $G/H\cong S^{2n-1}$, $G/K\cong P^n(\C)$, $K/H\cong \Ut(1)$. 
At the Lie algebra level we have that $\fh\subset \fk\subset\fg$ and we define subspaces $\fp,\fq$ of $\fg$ satisfying the orthogonal decomposition $\fg=\fk\oplus \fq $ and $\fk=\fh\oplus\fp$ with respect to the Killing form of $\fg$. 
The isotropy representation of $H$ is given by $\fp\oplus\fq$, with $\fp$ trivial, $\dim_\R\fp=1$, and $\fq$ equivalent to the standard representation, $\dim_\R\fq=2(n-1)$.

We denote by $\innerdots_0$ the inner product on $\fg$ given by $\inner{X}{Y}_0=-\frac12 \op{Re}(\tr(XY))$ for any $X,Y\in\fg=\ut(n)$. 
It turns out that $\innerdots_0$ is $\Ad(G)$-invariant. 
For $r,s>0$, we set 
\begin{equation}
\innerdots_{(r,s)}=\frac{1}{r^2} \innerdots_0|_{\fp} + \frac{1}{s^2}\innerdots_0|_{\fq}. 
\end{equation}
It turns out that the set of $G$-invariant metrics on $G/H=S^{2n-1}$ is $\{g_{(r,s)}:  r,s>0\}$.
The round metric with constant sectional curvature $1$ is $g_{(\frac{1}{\sqrt{2}},1)}$.
Of particular interest is $g_{(1,1)}$ for being the standard metric of $S^{2n-1}$ with respect to the realization $\Ut(n)/\Ut(n-1)$.

We next analyze the spectrum of the Laplace-Beltrami operator associated to the elliptic Riemannian manifold $(S^{2n-1},g_{(r,s)})$. 
Similarly as in \S\ref{subsec:Kohnelliptic}, the space of functions given by the restriction to $S^{2n-1}$ of polynomials in $\HH_{p,q}$ are eigenfunctions; in this case, the corresponding eigenvalue is (see \cite[Proof of Prop.~3.1]{BLPfullspec})
\begin{equation}
\lambda^{(p,q)}(r,s):=
\Big( 4\min(p,q)\big(\min(p,q)+|p-q|+n\big)+2|p-q|n \Big) s^2 + 2(p-q)^2r^2
.
\end{equation}

Let $M$ be a $(2n-1)$-dimensional elliptic manifold.
There is a finite subgroup $\Gamma$ of $\Ut(n)$ acting freely on $S^{2n-1}$ such that $M$ is diffeomorphic to $S^{2n-1}/\Gamma$. 
For $r,s>0$, we denote again by $g_{(r,s)}$ the metric on $S^{2n-1}/\Gamma$ induced by the finite cover $S^{2n-1} \to S^{2n-1}/\Gamma$, and we denote by $\Delta_{\Gamma}^{(r,s)}$ the Laplace-Beltrami operator associated to $(S^{2n-1}/\Gamma, g_{(r,s)})$. 
Similarly as in \S\ref{subsec:Kohnelliptic}, elements of $\HH_{p,q}^\Gamma$ induce eigenfunctions of $\Delta_{\Gamma}^{(r,s)}$ with eigenvalue $\lambda^{(p,q)}(r,s)$. 
Consequently, the multiplicity of a positive real number $\lambda$ in the spectrum of $\Delta_{\Gamma}^{(r,s)}$ is given by 
\begin{equation}
\sum_{{p,q\in\N_0 \,: \,  \lambda^{p,q}(r,s)=\lambda} } \dim\HH_{p,q}^\Gamma. 
\end{equation}

Here we state an immediate consequence of this description. 

\begin{theorem}
Let $\Gamma,\Gamma'$ be finite subgroups of $\Ut(n)$ acting freely on $S^{2n-1}$.
If $F_{\Gamma}(z,w)=F_{\Gamma'}(z,w)$, then the spectra of $\Delta_{\Gamma}^{(r,s)}$ and $\Delta_{\Gamma'}^{(r,s)}$ coincide, that is, the Riemannian manifolds $(S^{2n-1}/\Gamma, g_{(r,s)})$  and $(S^{2n-1}/\Gamma', g_{(r,s)})$ are isospectral, for any positive real numbers $r,s$. 
\end{theorem}

Consequently, all examples of finite subgroups $\Gamma,\Gamma'$ of $\Ut(n)$ acting freely on $S^{2n-1}$ satisfying that $F_{\Gamma}(z,w)=F_{\Gamma'}(z,w)$ from Sections~\ref{sec:ikeda}--\ref{sec:isospsphericalspaceforms} provide a curve of pairs of isospectral Riemannian manifolds. 
Although there are two parameters to move, $r$ and $s$, we obtain only a curve up to homotheties.

\bibliographystyle{plain}

\end{document}